\documentclass[11pt]{article}
\usepackage{amsmath, amssymb, amsthm}
\usepackage{graphicx}
\usepackage{tikz}
\usetikzlibrary{arrows, positioning, calc, arrows.meta, shapes.geometric, fit, backgrounds}
\usepackage{pgfplots}
\pgfplotsset{compat=1.17}
\usepackage[margin=2.5cm]{geometry}

\usepackage[hypertexnames=false]{hyperref}
\usepackage[hypcap=false]{caption}

\usepackage{enumitem}
\usepackage{booktabs}
\usepackage{xcolor}

\usepackage{float}

\newcommand{\E}{\mathbb{E}}
\newcommand{\PP}{\mathbb{P}}
\newcommand{\R}{\mathbb{R}}

\newtheorem{theorem}{Theorem}[section]
\newtheorem{proposition}{Proposition}[section]

\newtheorem{corollary}{Corollary}[section]
\newtheorem{definition}{Definition}[section]
\newtheorem{remark}{Remark}[section]

\title{Spectral Duality and Reset-Neutral Distributions\\
in Random Walks with Multi-Site Geometric Resetting}

\author{J.~A.~Vega Coso \\
        \small IUFyM, Universidad de Salamanca \\
        \small  Salamanca, Spain}

\date{}

\begin{document}

\maketitle

\begin{abstract}
We study the gambler's ruin problem for a biased random walk on
$\{0,1,\dots,a\}$ under multi-site geometric resetting: at each
time step, the walker is reset with probability $\gamma\in(0,1)$
to a random position drawn from a distribution $\pi$ over $m$
interior sites. Using renewal theory, we derive an exact
closed-form expression for the ruin probability $q_z(\gamma)$,
showing that the entire effect of $\pi$ is encoded in a single
scalar, the \emph{coupling constant}
$C(\pi,\gamma)=\bar{u}_\pi/\bar{s}_\pi$. A spectral analysis via
the Doob symmetrization then reveals the deep structure of this
coupling constant.

Our central result is a general criterion --- valid for any
absorbed Markov chain admitting a spectral decomposition --- for
the existence of a \emph{critical} (reset-neutral) distribution
$\pi^*$ for which $C(\pi^*,\gamma)$ is independent of $\gamma$.
The criterion is a \emph{spectral duality condition}: an involution
$\sigma$ on the reset sites with $\nu$-independent coupling weights
$\kappa(z)$ such that $B_\nu(z) = \kappa(z)\,A_\nu(\sigma(z))$ for
all spectral modes $\nu$, and with a common pair product
$\kappa(z)\kappa(\sigma(z))=K$. When this condition holds, the invariant
value is $C^*=1/(1+\sqrt K)$. For the biased random walk, the
spectral duality condition is equivalent to the geometric symmetry
$z_i + z_i' = a$, and $C^* = q_{a/2}^{(0)}$ is the classical ruin
probability from the midpoint, regardless of the specific choice of
symmetric reset sites or resetting rate. This result holds for any domain size $a$ (even
or odd), any number of reset sites $m$, and any bias $p\in(0,1)$.
Since the duality holds step by step in time as well as mode by
mode, the critical distributions are reset-neutral not only for
geometric resetting but for every resetting mechanism independent of
the walk under which absorption occurs almost surely.

Both algebraic and direct Monte Carlo verifications confirm the
theory to machine and statistical precision respectively, including
a stringent test of the predicted freedom of spectrally neutral
sites. Numerical illustrations also reveal a phase-like structure
in the space of reset distributions, with $\pi^*$ acting as a
separatrix between monotone regimes.
\end{abstract}

\section{Introduction}
\label{sec:intro}

\subsection{Random walks, absorption, and classical ruin theory}

Random walks constitute one of the cornerstones of probability theory and
statistical physics, providing a universal framework for modeling transport,
diffusion, and stochastic exploration. Since Pearson's early formulation of
random flight~\cite{Pearson1905} and Einstein's theory of Brownian
motion~\cite{Einstein1905}, random walk models have proven indispensable
across disciplines: from polymer conformations~\cite{DeGennes1979} and animal
foraging patterns~\cite{Viswanathan1996} to financial market
fluctuations~\cite{Mantegna2000} and algorithmic search~\cite{Lovasz1993}.
First-passage and absorption problems have played a central role in both
theory and applications~\cite{Redner2001, Condamin2007}.

Among the most classical examples is the \emph{gambler's ruin problem},
formalized in modern probabilistic language by Feller~\cite{Feller1968}.
Consider a biased random walk on $\{0,1,\dots,a\}$ with absorbing boundaries
at $0$ and $a$. Starting from position $z$, one asks for the probability
$q_z$ of eventual absorption at $0$ before reaching $a$. The classical
solution exhibits an exponential dependence on the bias ratio $q/p$ and has
found interpretations in areas ranging from finance and
insurance~\cite{Asmussen2010} to particle absorption in confined
geometries~\cite{Schuss2010} and decision-making under risk~\cite{Kahneman1979}.

\subsection{Stochastic resetting and multi-site dynamics}

In recent years, stochastic resetting has emerged as a powerful mechanism
for controlling random processes. The seminal work of Evans and
Majumdar~\cite{Evans2011} showed that diffusion with Poissonian resetting
reaches a non-equilibrium stationary state and can exhibit a finite optimal
mean first-passage time. This discovery initiated an extensive body of work
exploring resetting of L\'evy flights~\cite{Kusmierz2014}, search
problems~\cite{Reuveni2016}, reaction kinetics~\cite{Rotbart2015},
first passage under restart in general~\cite{Pal2017}, time-dependent
resetting protocols~\cite{Pal2016}, and nonequilibrium systems more
broadly~\cite{Evans2020}.

Absorption and splitting probabilities under resetting have also been
studied: for Brownian motion in an interval with two absorbing
endpoints, splitting probabilities under resetting have been
obtained~\cite{PalPrasad2019, Villarroel2022}. In discrete time,
spectral and renewal frameworks for resetting on networks and lattices
have been developed~\cite{Riascos2020, Riascos2022, Bonomo2021,
DasGiuggioli2022}, and exact splitting probabilities for absorbing
Markov chains under resetting have been derived~\cite{ChenLiHuang2022}.
The author has worked on two-barrier absorption problems for a
semi-deterministic walk with resetting~\cite{Villarroel2021} and,
without resetting, for compound Poisson processes with
drift~\cite{Villarroel2023a} and renewal processes with two-sided
jumps~\cite{Villarroel2023b}.

For a biased random walk on a finite interval reset to its starting
point, the ruin probability has an exact \emph{midpoint
invariance}~\cite{VegaCoso2026}: when the domain size $a$ is even, the
ruin probability from $z=a/2$ equals its classical, reset-free value
for every bias $p\in(0,1)$ and for every resetting mechanism
independent of the walk under which absorption occurs almost surely;
under geometric resetting this is $\partial q_{a/2}/\partial\gamma\equiv0$
for all $\gamma\in(0,1)$. Its origin is an exact duality between the
two absorption channels, which holds mode by mode in the spectral
representation and step by step in time~\cite{VegaCoso2026}.

In that result, as in most of the literature on absorption, the process
returns to a single fixed location. Resetting to several sites drawn
from a probability distribution $\pi$ has been studied on networks,
for transport and mean first-passage
properties~\cite{GonzalezRiascosBoyer2021, WangChenHuang2021}, but
splitting probabilities between two absorbing boundaries under
multi-site resetting have, to our knowledge, not been analysed. Restart
to several locations is natural in applied contexts, from search
strategies in foraging~\cite{Bartumeus2002, Benichou2011} to
diversified recovery protocols in risk management.

\subsection{Main results and organization}

In this paper we provide a complete analysis of the gambler's ruin problem
under multi-site geometric resetting and establish the following:

\begin{enumerate}[label=(\roman*)]
\item \textbf{Exact solution.} Using renewal theory, the ruin probability
      admits the closed-form expression
      \[
      q_z(\gamma) = u(z) + \bigl(1-s(z)\bigr)\,C(\pi,\gamma),
      \qquad C(\pi,\gamma) = \frac{\bar{u}_\pi}{\bar{s}_\pi},
      \]
      where $u(z)$, $s(z)$ are discounted first-passage probabilities
      and $C(\pi,\gamma)$ is a single scalar encoding all the information
      about the reset distribution (Section~\ref{sec:model}).

\item \textbf{Spectral representation.} A Doob $h$-transform diagonalizes
      the transition operator, yielding explicit spectral representations
      of $u(z)$ and $s(z)$ in terms of eigenvalues
      $\lambda_\nu=2\sqrt{pq}\cos(\pi\nu/a)$ and spectral coefficients
      $A_\nu(z)$, $B_\nu(z)$ (Section~\ref{sec:spectral}).

\item \textbf{General invariance principle.} A distribution $\pi^*$ with
      $C(\pi^*,\gamma)=\mathrm{const}$ exists whenever the spectral
      coefficients satisfy a \emph{duality condition} with $\nu$-independent
      weights, and for the biased random walk this condition is also
      necessary (Theorem~\ref{thm:general}, Section~\ref{sec:critical}).

\item \textbf{Critical distribution for the random walk.} For the biased
      random walk, the spectral duality condition is equivalent to the
      geometric symmetry $z_i+z_i'=a$ of the reset sites. The critical
      distribution satisfies $\pi_{z_i}^*/\pi_{z_i'}^*=(q/p)^{a/2-z_i}$,
      and the invariant value is $C^*=q_{a/2}^{(0)}$, the classical ruin
      probability from the midpoint (Corollary~\ref{cor:rw}).
      Because the duality also holds step by step in
      time~\cite{VegaCoso2026}, a walk reset to $\pi^*$ and started
      from $\pi^*$ (or, for even $a$, from the midpoint $a/2$) has
      ruin probability $C^*$ under every resetting
      mechanism independent of the walk
      (Proposition~\ref{prop:any_mechanism}).

\item \textbf{Numerical validation by two independent methods.} The
      predictions are confirmed (a) by direct algebraic solution of the
      renewal equations to machine precision, and (b) by Monte Carlo
      simulation of $10^7$ trajectories per configuration, agreeing
      within statistical error and providing a direct empirical test
      of the predicted freedom of spectrally neutral midpoint sites
      (Section~\ref{sec:numerics}).
\end{enumerate}

The paper is organized as follows. Section~\ref{sec:model} introduces the
model and derives the renewal equation. Section~\ref{sec:spectral} develops
the spectral representation. Section~\ref{sec:critical} establishes the
general invariance principle and its application to the random walk.
Section~\ref{sec:numerics} presents numerical illustrations. 
Section~\ref{sec:conclusion} summarizes findings and discusses extensions.

\section{Model and Renewal Equation}
\label{sec:model}

We consider a biased random walk on $E=\{0,1,\dots,a\}$ with absorbing
boundaries at $0$ and $a$, augmented by a stochastic resetting mechanism.

\paragraph{Notational convention.}
Throughout the paper, $q:=1-p$ denotes the step-down probability of the
walk, while $q_z(\gamma)$ (indexed by a starting position $z$) denotes the
ruin probability. No confusion should arise, as the former always appears
in ratios $(q/p)$ or exponents, and the latter is always subscripted by
a spatial variable.

\subsection{Formal definition}

\begin{definition}[Random walk with multi-site geometric resetting]
\label{def:rw_multi_reset}
Fix $a\geq 2$ and initial position $X_0=z\in\{1,\dots,a-1\}$. Let
$\{\xi_n\}$ be i.i.d.\ with $\PP(\xi_n=+1)=p$, $\PP(\xi_n=-1)=q=1-p$,
$p\in(0,1)$. Let $\{R_n\}$ be i.i.d.\ Bernoulli$(\gamma)$, $\gamma\in(0,1)$,
independent of $\{\xi_n\}$. Let $\{Z_n\}$ be i.i.d.\ with distribution
$\pi=(\pi_{z_1},\dots,\pi_{z_m})$ on $\{z_1,\dots,z_m\}\subset\{1,\dots,a-1\}$,
$\sum_i\pi_{z_i}=1$, independent of both sequences. The dynamics at each
$n\geq 1$ are:
\begin{enumerate}[label=(\roman*)]
\item If $R_n=1$: set $X_n=Z_n$ (reset event).
\item If $R_n=0$: set $X_n=X_{n-1}+\xi_n$; if $X_n\in\{0,a\}$, terminate.
\end{enumerate}
\end{definition}

\begin{remark}[Convention and reduction to single-site]
The reset decision precedes the walk step, consistent with the single-site
model of Ref.~\cite{VegaCoso2026}: at a reset step the walker is
relocated and does not move, so a reset is never itself an absorption
event. When $\pi=\delta_{z_0}$, the model reduces to single-site
resetting at $z_0$.
\end{remark}

Each reset event creates independent \emph{cycles}~\cite{Feller1968,
Asmussen2003}, providing the regenerative structure for the renewal analysis.
Figure~\ref{fig:trajectory} illustrates a sample trajectory.

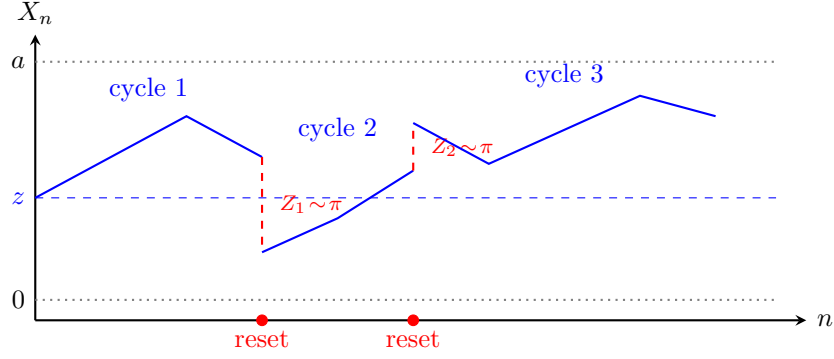
\begin{figure}[H]
\centering
\begin{tikzpicture}[x=1.0cm, y=0.9cm, >=stealth]
\draw[->, thick] (0,0) -- (10.2,0) node[right, font=\small] {$n$};
\draw[->, thick] (0,0) -- (0,4.2) node[above, font=\small] {$X_n$};
\draw[gray, dotted, thick] (0,0.3) -- (9.8,0.3);
\node[left, font=\small] at (0,0.3) {$0$};
\draw[gray, dotted, thick] (0,3.8) -- (9.8,3.8);
\node[left, font=\small] at (0,3.8) {$a$};
\draw[blue, dashed] (0,1.8) -- (9.8,1.8);
\node[left, blue, font=\small] at (0,1.8) {$z$};
\draw[thick, blue] (0,1.8) -- (1,2.4) -- (2,3.0) -- (3,2.4);
\draw[red, dashed, thick] (3,2.4) -- (3,1.0);
\filldraw[red] (3,0) circle (2pt) node[below, font=\small] {reset};
\node[red, right, font=\scriptsize] at (3.1,1.7) {$Z_1\!\sim\!\pi$};
\draw[thick, blue] (3,1.0) -- (4,1.5) -- (5,2.2);
\draw[red, dashed, thick] (5,2.2) -- (5,2.9);
\filldraw[red] (5,0) circle (2pt) node[below, font=\small] {reset};
\node[red, right, font=\scriptsize] at (5.1,2.55) {$Z_2\!\sim\!\pi$};
\draw[thick, blue] (5,2.9) -- (6,2.3) -- (7,2.8) -- (8,3.3) -- (9,3.0);
\node[blue, font=\small] at (1.5,3.4) {cycle 1};
\node[blue, font=\small] at (4.0,2.8) {cycle 2};
\node[blue, font=\small] at (7.0,3.6) {cycle 3};
\end{tikzpicture}
\caption{Sample trajectory under multi-site geometric resetting.
Starting from $z$, the walk evolves freely until a reset occurs
(red dashed arrows), after which it is relocated to an independent
draw $Z_i\sim\pi$. Each cycle evolves independently until either
absorption at a boundary or the next reset event.}
\label{fig:trajectory}
\end{figure}

\subsection{Auxiliary functions}

\begin{definition}[Discounted first-cycle probabilities]
\label{def:auxiliary}
For $z\in\{1,\dots,a-1\}$, let $T_{\mathrm{reset}}$ be the first reset time
(geometric with parameter $\gamma$, so $\PP(T_{\mathrm{reset}}>k)=(1-\gamma)^k$
for $k\geq 0$). Define:
\begin{align}
u(z) &:= \PP_z(\tau_0<\tau_a,\;\tau_0<T_{\mathrm{reset}})
       = \E_z\!\bigl[(1-\gamma)^{\tau_0}\mathbf{1}_{\{\tau_0<\tau_a\}}\bigr],
\label{eq:u_def}\\
s(z) &:= \PP_z(\tau<T_{\mathrm{reset}})
       = \E_z\!\bigl[(1-\gamma)^{\tau}\bigr],
\label{eq:s_def}
\end{align}
where $\tau=\min\{\tau_0,\tau_a\}$. The rightmost equalities follow from
the independence of $T_{\mathrm{reset}}$ and $\{\xi_n\}$ via
\[
u(z) = \sum_{k=1}^\infty \PP_z(\tau_0=k,\,\tau_0<\tau_a)\,\PP(T_{\mathrm{reset}}>k)
     = \sum_{k=1}^\infty \PP_z(\tau_0=k,\,\tau_0<\tau_a)(1-\gamma)^k,
\]
and analogously for $s(z)$.
\end{definition}

Here $u(z)$ is the probability of ruin in the first cycle (before any reset),
and $s(z)$ is the probability that the first cycle ends with absorption.
Since $\tau\geq 1$, we have $s(z)\leq 1-\gamma<1$.

The functions $u(z)$ and $s(z)$ satisfy the recurrences
\begin{align}
u(z) &= (1-\gamma)\bigl[p\,u(z+1)+q\,u(z-1)\bigr],
\label{eq:u_recursion}\\
s(z) &= (1-\gamma)\bigl[p\,s(z+1)+q\,s(z-1)\bigr],
\label{eq:s_recursion}
\end{align}
with boundary conditions $u(0)=s(0)=1$, $u(a)=0$, $s(a)=1$.
These tridiagonal systems are solved by fixed-point iteration, which
converges to tolerance $10^{-12}$ in fewer than $10^4$ steps for all
parameter values considered here. Monte Carlo validation confirms relative
errors below $0.4\%$ across all tested configurations.

\subsection{Renewal equation and closed-form solution}

Conditioning on the first cycle:

\begin{theorem}[Exact solution]
\label{thm:multisite_solution}
The ruin probability under multi-site geometric resetting is
\begin{equation}
\boxed{
q_z(\gamma) = u(z) + \bigl(1-s(z)\bigr)\,C(\pi,\gamma),
\qquad
C(\pi,\gamma) = \frac{\bar{u}_\pi}{\bar{s}_\pi},
}
\label{eq:solution_formula}
\end{equation}
where $\bar{u}_\pi=\sum_i\pi_{z_i}u(z_i)$ and
$\bar{s}_\pi=\sum_i\pi_{z_i}s(z_i)>0$.
\end{theorem}

\begin{proof}
By the strong Markov property at $T_{\mathrm{reset}}$ and the independence
of $Z_n\sim\pi$ from both $\{\xi_k\}$ and $\{R_k\}$, the process after the
first reset is an independent copy of the original process with initial
position $Z_1\sim\pi$. Conditioning on the outcome of the first cycle
(absorption at $0$, absorption at $a$, or reset) therefore gives
\[
q_z(\gamma) = u(z) + \bigl(1-s(z)\bigr)\sum_{i=1}^m \pi_{z_i}\,q_{z_i}(\gamma).
\]
Set $C := \sum_i \pi_{z_i}\,q_{z_i}(\gamma)$. Evaluating the previous
identity at each reset site $z_k$ and summing against $\pi_{z_k}$ yields
\[
C = \bar{u}_\pi + (1-\bar{s}_\pi)\,C,
\]
whence $C=\bar{u}_\pi/\bar{s}_\pi$, provided $\bar{s}_\pi>0$. The latter
is immediate since $\bar s_\pi=\sum_i\pi_{z_i}s(z_i)$ and each
$s(z_i)>0$ by~\eqref{eq:s_def}.
\end{proof}

\begin{remark}[Dimensional reduction]
Although nominally $m$ unknowns $\{q_{z_i}\}$ appear, the renewal structure
collapses them to the single scalar $C(\pi,\gamma)$, regardless of $m$. This
scalar is the central object of analysis in Section~\ref{sec:critical}.
\end{remark}

\begin{remark}[Connection to Feynman--Kac theory]
Setting $\lambda=-\log(1-\gamma)$, the functions $u(z)=\E_z[e^{-\lambda\tau_0}
\mathbf{1}_{\{\tau_0<\tau_a\}}]$ and $s(z)=\E_z[e^{-\lambda\tau}]$ are
Laplace transforms of first-passage distributions. The
recurrences~\eqref{eq:u_recursion}--\eqref{eq:s_recursion} are the
corresponding discrete Feynman--Kac equations, providing a natural
framework for the spectral analysis of Section~\ref{sec:spectral}.
\end{remark}

\section{Spectral Representation}
\label{sec:spectral}

The renewal formula~\eqref{eq:solution_formula} expresses $q_z(\gamma)$ in
terms of $u(z)$ and $s(z)$. To analyze the structure of $C(\pi,\gamma)$,
we derive explicit spectral representations of these functions. This
spectral analysis plays a purely analytical role: it provides the foundation
for the invariance results of Section~\ref{sec:critical}.

\subsection{Doob $h$-transform and spectral decomposition}
\label{subsec:doob}

The transition operator $Q$ of the biased walk is not self-adjoint under the
standard inner product. As a nearest-neighbour chain it is reversible,
and the natural setting for its spectral analysis is a weighted
Hilbert space~\cite{KarlinMcGregor1959}. We symmetrize it by the
diagonal similarity transform that we call, as in
Ref.~\cite{VegaCoso2026}, the \emph{Doob
transform}~\cite{Doob1953, ChetriteTouchette2015} (in the sense of a
spectral conjugation, without interpreting the conjugated operator as a
stochastic matrix), with weight
\begin{equation}
h(x) := \left(\frac{q}{p}\right)^{x/2}, \qquad x\in\{1,\dots,a-1\}.
\label{eq:weight_h}
\end{equation}
The conjugated operator $\widetilde{Q}:=D^{-1}QD$ with
$D=\mathrm{diag}(h(1),\dots,h(a-1))$ has entries
$\widetilde{Q}(x,x\pm 1)=\sqrt{pq}$ and is self-adjoint on
$\ell^2(\{1,\dots,a-1\})$. Spectral methods have been applied to random
walks under resetting on arbitrary networks~\cite{Riascos2020,
Riascos2022}, and closed-form propagators for confined biased lattice
walks, without resetting, are given in~\cite{SarvaharmanGiuggioli2020}.

\begin{proposition}[Spectral decomposition]
\label{prop:eigenvalues}
The operator $\widetilde{Q}$ has eigenvalues
$\lambda_\nu=2\sqrt{pq}\cos(\pi\nu/a)$, $\nu=1,\dots,a-1$,
with orthonormal eigenfunctions $\psi_\nu(x)=\sqrt{2/a}\,\sin(\pi\nu x/a)$.
\end{proposition}

\begin{proof}
The matrix $\widetilde{Q}$ is symmetric tridiagonal with off-diagonal
entries $\sqrt{pq}$. Its eigenvectors are discrete sine functions and
the eigenvalues follow from the dispersion relation of the discrete
Laplacian~\cite{Feller1968, KarlinMcGregor1959}. \qed
\end{proof}

\subsection{Spectral formulas for $u(z)$ and $s(z)$}
\label{subsec:spectral_u_s}

\begin{theorem}[Spectral representations]
\label{thm:spectral_u_s}
For any $z\in\{1,\dots,a-1\}$ and $\gamma\in(0,1)$:
\begin{align}
u(z) &= \sum_{\nu=1}^{a-1} A_\nu(z)\,f_\nu(\gamma),
\label{eq:u_spectral}\\
s(z) &= \sum_{\nu=1}^{a-1} \bigl[A_\nu(z)+B_\nu(z)\bigr]\,f_\nu(\gamma),
\label{eq:s_spectral}
\end{align}
where the spectral coefficients are
\begin{align}
A_\nu(z) &= \frac{2}{a}\left(\frac{q}{p}\right)^{z/2}
\sin\!\left(\frac{\pi\nu z}{a}\right)\sin\!\left(\frac{\pi\nu}{a}\right),
\label{eq:A_nu}\\
B_\nu(z) &= \frac{2}{a}\left(\frac{p}{q}\right)^{(a-z)/2}
\sin\!\left(\frac{\pi\nu(a-z)}{a}\right)\sin\!\left(\frac{\pi\nu}{a}\right),
\label{eq:B_nu}
\end{align}
and the spectral transfer function is
\begin{equation}
f_\nu(\gamma) := \frac{\sqrt{pq}\,(1-\gamma)}{1-\lambda_\nu(1-\gamma)}.
\label{eq:f_nu}
\end{equation}
\end{theorem}

\begin{proof}[Sketch]
Write $u(z)=\sum_{k\geq 1}(1-\gamma)^k\PP_z(\tau_0=k,\tau_0<\tau_a)$.
Absorption at $0$ at step $k$ requires being at site $1$ at step $k-1$
and then stepping left, so
$\PP_z(\tau_0=k,\tau_0<\tau_a)=q\,[Q^{k-1}]_{z,1}$, and diagonalizing
$Q$ via the Doob transform yields
$q\,[Q^{k-1}]_{z,1}=\sqrt{pq}\sum_\nu A_\nu(z)\lambda_\nu^{k-1}$
(Ref.~\cite{VegaCoso2026}, Theorem on channel amplitudes). Summing the
resulting geometric series in $(1-\gamma)\lambda_\nu$, which converges
since $|\lambda_\nu|(1-\gamma)<1$, gives~\eqref{eq:u_spectral}. The
derivation of the absorption-at-$a$ part of~\eqref{eq:s_spectral},
through $p\,[Q^{k-1}]_{z,a-1}=\sqrt{pq}\sum_\nu B_\nu(z)\lambda_\nu^{k-1}$,
is analogous. \qed
\end{proof}

\begin{remark}
$A_\nu(z)$ and $B_\nu(z)$ are the spectral projections of the initial
condition $z$ onto eigenmode $\nu$, weighted by the boundary fluxes at
$0$ and $a$ respectively. The Doob factors $(q/p)^{z/2}$ and
$(p/q)^{(a-z)/2}$ encode the bias asymmetry.
\end{remark}

\subsection{Crucial identity and linear independence}
\label{subsec:key_identity}

\begin{proposition}[Spectral duality identity]
\label{prop:duality}
For all $\nu=1,\dots,a-1$ and all $z\in\{1,\dots,a-1\}$,
\begin{equation}
B_\nu(z) = \left(\frac{p}{q}\right)^{a-z} A_\nu(a-z).
\label{eq:B_from_A}
\end{equation}
\end{proposition}

\begin{proof}
Using $\sin(\pi\nu(a-z)/a)=\sin(\pi\nu-\pi\nu z/a)=(-1)^{\nu+1}
\sin(\pi\nu z/a)$ and substituting into~\eqref{eq:B_nu}:
\[
B_\nu(z) = \frac{2}{a}\left(\frac{p}{q}\right)^{(a-z)/2}
(-1)^{\nu+1}\sin\!\left(\frac{\pi\nu z}{a}\right)
\sin\!\left(\frac{\pi\nu}{a}\right).
\]
On the other hand, from~\eqref{eq:A_nu} evaluated at $a-z$:
\[
A_\nu(a-z) = \frac{2}{a}\left(\frac{q}{p}\right)^{(a-z)/2}
\sin\!\left(\frac{\pi\nu(a-z)}{a}\right)\sin\!\left(\frac{\pi\nu}{a}\right)
= \frac{2}{a}\left(\frac{q}{p}\right)^{(a-z)/2}
(-1)^{\nu+1}\sin\!\left(\frac{\pi\nu z}{a}\right)
\sin\!\left(\frac{\pi\nu}{a}\right).
\]
Hence $(p/q)^{a-z}A_\nu(a-z) = \frac{2}{a}(p/q)^{(a-z)/2}
(-1)^{\nu+1}\sin(\pi\nu z/a)\sin(\pi\nu/a) = B_\nu(z)$. \qed
\end{proof}

\begin{proposition}[Linear independence]
\label{prop:indep_A}
Let $z_1,\dots,z_m\in\{1,\dots,a-1\}$ be distinct with $m\leq a-1$.
The vectors $v_i:=(A_\nu(z_i))_{\nu=1}^{a-1}\in\R^{a-1}$ are linearly
independent.
\end{proposition}

\begin{proof}
The matrix with entries $W_{\nu i}=\sin(\pi\nu z_i/a)$ is a submatrix of
the discrete sine-transform matrix $S_{\nu i}=\sin(\pi\nu i/a)$, which
satisfies $S^\top S = \frac{a}{2}I$ and is therefore invertible. Since
the $z_i$ are distinct, this submatrix has full column rank. Multiplying
row $\nu$ by $\sin(\pi\nu/a)\neq 0$ and column $i$ by
$(q/p)^{z_i/2}\neq 0$ preserves linear independence. \qed
\end{proof}

\begin{remark}[Linear independence justifies mode-by-mode analysis]
The functions $\{f_\nu(\gamma)\}_{\nu=1}^{a-1}$ are linearly independent
over $(0,1)$. Writing $t=1-\gamma$, they are $\sqrt{pq}\,t/(1-\lambda_\nu t)$
with pairwise distinct $\lambda_\nu$. If $\sum_\nu c_\nu f_\nu\equiv0$
on $(0,1)$, then the rational function
$\sum_\nu c_\nu/(1-\lambda_\nu t)$ vanishes identically; each term with
$\lambda_\nu\neq0$ has a simple pole at $t=\lambda_\nu^{-1}$ that no other
term cancels, so $c_\nu=0$ for $\lambda_\nu\neq0$, and then the
remaining term (the constant one, $\lambda_{a/2}=0$, present only for
even $a$) also has $c_{a/2}=0$. Together with Proposition~\ref{prop:indep_A},
this justifies hypotheses~(H1)--(H2) of Section~\ref{sec:critical}.
Note that $m\leq a-1$ is automatic since there are only $a-1$ interior
sites available.
\end{remark}

\section{Critical reset distribution and invariance of $C$}
\label{sec:critical}

The renewal equation~\eqref{eq:solution_formula} shows that the
entire effect of the reset distribution $\pi$ on the ruin
probability $q_z(\gamma)$ is mediated through the single scalar
\begin{equation}
C(\pi,\gamma) = \frac{\bar{u}_\pi(\gamma)}{\bar{s}_\pi(\gamma)}
= \frac{\displaystyle\sum_{i=1}^m \pi_{z_i}\,u(z_i;\gamma)}
       {\displaystyle\sum_{i=1}^m \pi_{z_i}\,s(z_i;\gamma)}.
\label{eq:C_def_recall}
\end{equation}
In general, $C$ depends on both $\pi$ and $\gamma$. The central
question is: does there exist a distribution $\pi^*$ for which
$C(\pi^*,\gamma)$ is independent of $\gamma$? Such a distribution
would constitute a \emph{reset-neutral} strategy: the coupling
constant remains the same regardless of how frequently resets occur.

We answer this question at two levels. In
Section~\ref{sec:general_principle} we establish a structural
criterion valid for any absorbed Markov chain admitting a
spectral decomposition — the result is not specific to random
walks but holds whenever the spectral coefficients satisfy a
duality condition. In Section~\ref{sec:application} we
instantiate this criterion for the biased random walk on
$\{0,\dots,a\}$, where the spectral duality reduces to a
geometric symmetry of the reset sites, and we derive explicit
formulas.

\subsection{Spectral duality and projective invariance}
\label{sec:general_principle}

We work in the general setting of a Markov chain absorbed at
the boundary of a finite domain, with a spectral decomposition
of the form
\begin{equation}
u(z;\gamma) = \sum_{\nu=1}^{N} A_\nu(z)\,f_\nu(\gamma),
\qquad
s(z;\gamma) = \sum_{\nu=1}^{N} [A_\nu(z)+B_\nu(z)]\,f_\nu(\gamma),
\label{eq:spectral_abstract}
\end{equation}
where $A_\nu(z)$ and $B_\nu(z)$ are spectral coefficients
encoding the ``ruin'' and ``survival'' components of mode $\nu$
at site $z$, and $\{f_\nu(\gamma)\}_{\nu=1}^N$ are functions
of the resetting rate. We impose two standing hypotheses:
\begin{enumerate}[label=\textbf{(H\arabic*)}]
\item \label{hyp:f_indep} The functions $\{f_\nu(\gamma)\}_{\nu=1}^N$
      are linearly independent over $(0,1)$.
\item \label{hyp:A_indep} For any finite set of $m\leq N$ distinct sites
      $z_1,\dots,z_m$, the vectors $(A_\nu(z_i))_{\nu=1}^N
      \in\mathbb{R}^N$ are linearly independent: any identity
      of the form $\sum_i c_i A_\nu(z_i)=0$ for all
      $\nu=1,\dots,N$ forces $c_i=0$ for all $i$.
\end{enumerate}
These hypotheses hold for the biased random walk, as verified
in Section~\ref{sec:spectral} (Proposition~\ref{prop:indep_A}).

The key structural notion is the following.

\begin{definition}[Spectral duality with $\nu$-independent weights]
\label{def:spectral_duality}
Let $\{z_1,\dots,z_m\}$ be a set of sites. A
\emph{spectral duality} on this set is an involution
$\sigma:\{z_1,\dots,z_m\}\to\{z_1,\dots,z_m\}$
(i.e., $\sigma^2=\mathrm{id}$) together with positive
coefficients $\kappa(z)>0$, independent of $\nu$, such that
\begin{equation}
B_\nu(z) = \kappa(z)\,A_\nu(\sigma(z))
\qquad \forall\,\nu=1,\dots,N,\quad \forall\,z\in\{z_1,\dots,z_m\}.
\label{eq:spectral_duality}
\end{equation}
A site $z_0$ with $\sigma(z_0)=z_0$ is called a
\emph{spectrally neutral site}. A pair $(z,\sigma(z))$ with
$\sigma(z)\neq z$ is called a \emph{spectral pair}.
\end{definition}

\begin{remark}[Interpretation]
\label{rem:duality_interpretation}
Condition~\eqref{eq:spectral_duality} says that the involution
$\sigma$ exchanges the two spectral channels --- ``ruin''
($A_\nu$) and ``survival'' ($B_\nu$) --- up to a multiplicative
weight $\kappa(z)$ that does not depend on the mode $\nu$.
This is a \emph{twisted symmetry}: it acts simultaneously on
the space of sites (via $\sigma$) and on the spectral channels
(via $\kappa$). The $\nu$-independence of $\kappa$ is the
essential condition: it allows the mode-by-mode constraints
to decouple, making the system solvable.
\end{remark}

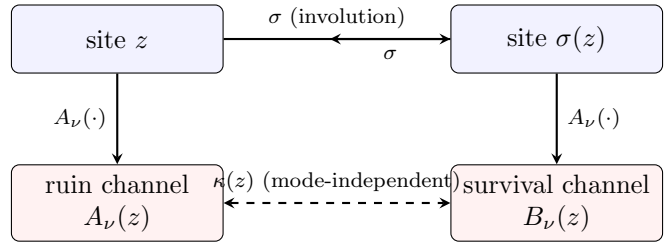
\begin{figure}[h!]
\centering
\begin{tikzpicture}[node distance=2.2cm, >=stealth,
    box/.style={draw, rounded corners=3pt, minimum width=2.8cm,
                minimum height=0.9cm, align=center, font=\small}]
\node[box, fill=blue!5]  (z)     {site $z$};
\node[box, fill=blue!5]  (sz)    [right=3cm of z]    {site $\sigma(z)$};
\node[box, fill=red!5]   (chA)   [below=1.2cm of z]  {ruin channel\\$A_\nu(z)$};
\node[box, fill=red!5]   (chB)   [below=1.2cm of sz] {survival channel\\$B_\nu(z)$};
\draw[->, thick] (z)   -- node[above, font=\scriptsize] {$\sigma$ (involution)} (sz);
\draw[->, thick] (sz)  -- node[below, font=\scriptsize] {$\sigma$} ++(-3cm,0);
\draw[->, thick] (z)   -- node[left, font=\scriptsize] {$A_\nu(\cdot)$} (chA);
\draw[->, thick] (sz)  -- node[right, font=\scriptsize] {$A_\nu(\cdot)$} (chB);
\draw[<->, thick, dashed] (chA) -- node[above, font=\scriptsize] {$\kappa(z)$ (mode-independent)} (chB);
\end{tikzpicture}
\caption{Twisted symmetry structure of the spectral duality condition.
The involution $\sigma$ pairs sites; the coefficient $\kappa(z)$ couples
the ruin and survival spectral channels with a weight that does not
depend on the mode index $\nu$.}
\label{fig:twisted_symmetry}
\end{figure}

\begin{theorem}[Spectral duality and projective invariance]
\label{thm:general}
Assume hypotheses~\ref{hyp:f_indep}--\ref{hyp:A_indep}.
Let $\{z_1,\dots,z_m\}$ be reset sites admitting a spectral
duality $(\sigma,\kappa)$ in the sense of
Definition~\ref{def:spectral_duality}, with spectral pairs
$(z_i,z_i')$, $i=1,\dots,k$, and spectrally neutral sites
$z_0^{(j)}$, $j=1,\dots,\ell$. Suppose further that:
\begin{enumerate}[label=\textbf{(H\arabic*)},,resume]
\item \label{hyp:C_consistent} The product $\kappa(z_i)\kappa(z_i')$
      is the same for all spectral pairs $i=1,\dots,k$.
      Denote this common value by $K$. This ensures global
      compatibility of the pairwise constraints, guaranteeing
      that a single value $C^*$ satisfies all pair equations
      simultaneously.
\item \label{hyp:neutral_consistent} For each neutral site
      $z_0^{(j)}$, the consistency condition
      $\kappa(z_0^{(j)}) = \sqrt{K}$ holds.
\end{enumerate}
Then the following hold.
\begin{enumerate}[label=\emph{(\roman*)}]
\item \emph{Existence and characterization.}
      There exists a family of distributions $\pi^*$ on
      $\{z_1,\dots,z_m\}$ such that $C(\pi^*,\gamma)$ is
      independent of $\gamma$ for all $\gamma\in(0,1)$.
      The family is characterized by
      \begin{equation}
      \frac{\pi_{z_i}^*}{\pi_{z_i'}^*}
      = \sqrt{\frac{\kappa(z_i')}{\kappa(z_i)}},
      \qquad i=1,\dots,k,
      \label{eq:pi_star_general}
      \end{equation}
      with each neutral weight $\pi_{z_0^{(j)}}^*$
      completely free, subject only to
      $\sum_i(\pi_{z_i}^*+\pi_{z_i'}^*)+\sum_j\pi_{z_0^{(j)}}^*=1$.

\item \emph{Projective invariance.}
      For any $\pi^*$ satisfying~\eqref{eq:pi_star_general},
      the invariant value of $C$ is
      \begin{equation}
      C^* := C(\pi^*,\gamma)
      = \frac{1}{1+\sqrt{K}},
      \label{eq:C_star_general}
      \end{equation}
      independent of $\gamma$, of the specific choice of
      $\pi^*$ within the family, and of the neutral weights
      $\pi_{z_0^{(j)}}^*$. In the projective space
      $\mathbb{P}^1$ generated by the spectral channels
      $(A_\nu,B_\nu)$, the distribution $\pi^*$ defines
      the fixed point $[C^*:1-C^*]$, independent of $\nu$.

\item \emph{Necessity (biased random walk).}
      For the biased random walk of Section~\ref{sec:spectral},
      conversely, suppose $\pi$ is a distribution with
      $\operatorname{supp}(\pi)=\{z_1,\dots,z_m\}$ such that
      $C(\pi,\gamma)$ is independent of $\gamma$. Then the set
      $\{z_1,\dots,z_m\}$ admits a spectral duality satisfying
      \ref{hyp:C_consistent}--\ref{hyp:neutral_consistent}.
      In particular, if the sites $\{z_1,\dots,z_m\}$ admit no
      spectral duality, no such $\pi$ with full support on the
      set can exist.
\end{enumerate}
\end{theorem}

\begin{proof}
\textbf{Sufficiency~(i)--(ii).}
The condition $C(\pi,\gamma)=C^*$ for all $\gamma$ is
equivalent to
\begin{equation}
\bar{u}_\pi(\gamma) = C^*\,\bar{s}_\pi(\gamma)
\qquad \forall\,\gamma\in(0,1).
\label{eq:proportionality}
\end{equation}
Substituting~\eqref{eq:spectral_abstract}:
\begin{align*}
\bar{u}_\pi(\gamma)
&= \sum_{\nu=1}^N
   \Bigl(\sum_{i=1}^m \pi_{z_i} A_\nu(z_i)\Bigr) f_\nu(\gamma),\\
\bar{s}_\pi(\gamma)
&= \sum_{\nu=1}^N
   \Bigl(\sum_{i=1}^m \pi_{z_i}[A_\nu(z_i)+B_\nu(z_i)]\Bigr)
   f_\nu(\gamma).
\end{align*}
By hypothesis~\ref{hyp:f_indep}, condition~\eqref{eq:proportionality}
holds for all $\gamma$ if and only if it holds mode by mode:
\begin{equation}
\sum_{i=1}^m \pi_{z_i}(1-C^*)A_\nu(z_i)
= C^*\sum_{i=1}^m \pi_{z_i} B_\nu(z_i)
\qquad \forall\,\nu.
\label{eq:mode_general}
\end{equation}
Substituting the spectral duality~\eqref{eq:spectral_duality}:
\[
\sum_{i=1}^m \pi_{z_i}(1-C^*)A_\nu(z_i)
= C^*\sum_{i=1}^m \pi_{z_i}\kappa(z_i) A_\nu(\sigma(z_i)).
\]
Since $\sigma$ is a bijection and an involution, the
change of variables $z_i\mapsto\sigma(z_i)$ is a
permutation of the index set, and we may rewrite the
right-hand side as:
\[
C^*\sum_{i=1}^m \pi_{\sigma(z_i)}\kappa(\sigma(z_i)) A_\nu(z_i).
\]
By hypothesis~\ref{hyp:A_indep}, the sequences
$\{A_\nu(z_i)\}_\nu$ are linearly independent for distinct
sites, so we may equate coefficients of $A_\nu(z_i)$
for each $i$:
\begin{equation}
\pi_{z_i}(1-C^*) = C^*\,\kappa(\sigma(z_i))\,\pi_{\sigma(z_i)}
\qquad \forall\,i.
\label{eq:coeff_general}
\end{equation}

\textit{Spectral pairs.}
For a pair $(z_i,z_i')$ with $\sigma(z_i)=z_i'$,
equation~\eqref{eq:coeff_general} gives:
\begin{align}
\pi_{z_i}(1-C^*) &= C^*\kappa(z_i')\,\pi_{z_i'},
\label{eq:pair_1}\\
\pi_{z_i'}(1-C^*) &= C^*\kappa(z_i)\,\pi_{z_i}.
\label{eq:pair_2}
\end{align}
Multiplying~\eqref{eq:pair_1} and~\eqref{eq:pair_2}:
\[
(1-C^*)^2 = C^{*2}\kappa(z_i)\kappa(z_i') = C^{*2}K,
\]
where hypothesis~\ref{hyp:C_consistent} ensures $K$ is the
same for all pairs. Taking the positive square root
(since $C^*\in(0,1)$ implies $1-C^*>0$):
\[
1-C^* = C^*\sqrt{K}
\quad\Longleftrightarrow\quad
C^*(1+\sqrt{K})=1,
\]
so that
\begin{equation}
C^* = \frac{1}{1+\sqrt{K}},
\label{eq:C_star_proof}
\end{equation}
proving~\eqref{eq:C_star_general}.
Dividing~\eqref{eq:pair_1} by~\eqref{eq:pair_2}:
\[
\frac{\pi_{z_i}^*}{\pi_{z_i'}^*}
= \frac{\kappa(z_i')}{\kappa(z_i)}\cdot
\left(\frac{\pi_{z_i'}^*}{\pi_{z_i}^*}\right)^{-1},
\]
which gives $(\pi_{z_i}^*/\pi_{z_i'}^*)^2 =
\kappa(z_i')/\kappa(z_i)$, hence~\eqref{eq:pi_star_general}.

\textit{Spectrally neutral sites.}
For a neutral site $\sigma(z_0)=z_0$,
equation~\eqref{eq:coeff_general} gives:
\[
\pi_{z_0}(1-C^*) = C^*\kappa(z_0)\,\pi_{z_0}.
\]
If $\pi_{z_0}>0$, this requires $1-C^*=C^*\kappa(z_0)$,
i.e., $\kappa(z_0)=(1-C^*)/C^*$. Using~\eqref{eq:C_star_proof},
$(1-C^*)/C^*=\sqrt{K}$, so that $\kappa(z_0)=\sqrt{K}$ ---
hypothesis~\ref{hyp:neutral_consistent}. When this
holds, the equation for $\pi_{z_0}$ is automatically
satisfied for any value of $\pi_{z_0}\geq 0$, so
$\pi_{z_0}^*$ is completely free.

\textit{Projective interpretation.}
The invariant value $C^*=1/(1+\sqrt{K})$ is
independent of $\gamma$, of the choice of $\pi^*$ within
the family, and of the neutral weights. In $\mathbb{P}^1$,
the ratio $\bar{u}_\pi:\bar{s}_\pi - \bar{u}_\pi$
equals $C^*:(1-C^*)$ mode by mode, so the spectral
projection of $\pi^*$ is the fixed point $[C^*:1-C^*]$,
independent of $\nu$.

\medskip
\textbf{Necessity~(iii).}
Suppose that a distribution $\pi$ with $\operatorname{supp}(\pi)
=\{z_1,\dots,z_m\}$ (i.e., $\pi_{z_i}>0$ for all $i$) satisfies
$C(\pi,\gamma)\equiv C^*$ for all $\gamma\in(0,1)$. We show that
the reset sites admit a spectral duality $(\sigma,\kappa)$ satisfying
\ref{hyp:C_consistent}--\ref{hyp:neutral_consistent}, with the
value $K$ such that $C^*=1/(1+\sqrt{K})$.

By hypothesis~\ref{hyp:f_indep}, $C\equiv C^*$ forces the mode-by-mode
identity~\eqref{eq:mode_general}. Here $N=a-1$, and by
Proposition~\ref{prop:indep_A} the $a-1$ vectors
$(A_\nu(z))_{\nu=1}^{N}$, $z=1,\dots,a-1$, form a basis of $\R^N$. Let
$v_i=(A_\nu(z_i))_{\nu}$ for the reset sites, and complete
$\{v_1,\dots,v_m\}$ to this basis by the vectors
$w_y=(A_\nu(y))_\nu$, $y\notin\{z_1,\dots,z_m\}$. Each $B_\nu(z_i)$, viewed as a function of $\nu\in\{1,\dots,N\}$,
decomposes uniquely as
\begin{equation}
B_\nu(z_i) = \sum_{j=1}^m c_{ij}\,A_\nu(z_j) + r_i(\nu),
\qquad r_i\in\operatorname{span}\{w_y:\ y\notin\{z_1,\dots,z_m\}\}.
\label{eq:B_decomp}
\end{equation}
Substituting into~\eqref{eq:mode_general}:
\[
(1-C^*)\sum_j \pi_{z_j}\,A_\nu(z_j)
= C^*\sum_{i,j}\pi_{z_i}c_{ij}\,A_\nu(z_j) + C^*\sum_i\pi_{z_i}r_i(\nu).
\]
The left-hand side and the first sum on the right lie in
$\operatorname{span}\{v_1,\dots,v_m\}$, while $\sum_i\pi_{z_i}r_i(\nu)$
lies in the complementary subspace. By uniqueness of the direct-sum
decomposition:
\begin{equation}
\sum_{i=1}^m \pi_{z_i}\,r_i(\nu) = 0 \qquad \forall\,\nu,
\label{eq:rest_zero}
\end{equation}
and, comparing coefficients of $A_\nu(z_j)$,
\begin{equation}
\pi_{z_j}(1-C^*) = C^*\sum_{i=1}^m \pi_{z_i}\,c_{ij}
\qquad \forall\,j=1,\dots,m.
\label{eq:system_pi}
\end{equation}

\smallskip\noindent
\emph{Step 1: the matrix $M=(c_{ij})$ has at most one non-zero entry per row.}
By the closed-form identity~\eqref{eq:B_from_A} (the spectral identity
of the random walk), $B_\nu(z_i)=(p/q)^{a-z_i}A_\nu(a-z_i)$ is a multiple
of a single basis vector, the one attached to the site $a-z_i$. Hence,
if $a-z_i\in\{z_1,\dots,z_m\}$, the only non-zero $c_{ij}$ is the one
with $z_j=a-z_i$ and $r_i=0$; if $a-z_i\notin\{z_1,\dots,z_m\}$, then
all $c_{ij}=0$ and $r_i=(p/q)^{a-z_i}\,w_{a-z_i}\neq 0$.

\smallskip\noindent
\emph{Step 2: $r_i\equiv 0$ for all $i$.}
Let $I$ be the set of indices $i$ with $a-z_i\notin\{z_1,\dots,z_m\}$.
By Step~1, \eqref{eq:rest_zero} reads
$\sum_{i\in I}\pi_{z_i}(p/q)^{a-z_i}\,w_{a-z_i}=0$. The map
$z\mapsto a-z$ is injective, so the vectors $w_{a-z_i}$, $i\in I$, are
distinct elements of a basis, and all their coefficients vanish. Since
$\pi_{z_i}>0$ and $(p/q)^{a-z_i}>0$, this forces $I=\emptyset$: the site
set is closed under $z\mapsto a-z$, and $r_i\equiv0$ for every $i$.

\smallskip\noindent
\emph{Step 3: definition of $\sigma$ and $\kappa$.}
For each $i$, let $j(i)$ be the unique index with $z_{j(i)}=a-z_i$, and set
\[
\sigma(z_i):=z_{j(i)}=a-z_i, \qquad \kappa(z_i):=c_{i,j(i)}=(p/q)^{a-z_i}>0.
\]
Since $\sigma^2=\mathrm{id}$ and~\eqref{eq:B_from_A} gives
$B_\nu(z_i)=\kappa(z_i)A_\nu(\sigma(z_i))$, $(\sigma,\kappa)$ is a spectral
duality in the sense of Definition~\ref{def:spectral_duality}.

\smallskip\noindent
\emph{Step 4: (H3)--(H4) follow automatically.}
For any pair $(z_i,z_i'=a-z_i)$, $\kappa(z_i)\kappa(z_i')=(p/q)^{a-z_i}
\cdot(p/q)^{z_i}=(p/q)^a$, common to all pairs; this is
\ref{hyp:C_consistent} with $K=(p/q)^a$. For a neutral site
$z_0=a/2$, $\kappa(z_0)=(p/q)^{a/2}=\sqrt K$, which is
\ref{hyp:neutral_consistent}.

Finally, once $(\sigma,\kappa)$ is constructed, \eqref{eq:system_pi}
reduces to~\eqref{eq:coeff_general}, and the sufficiency analysis yields
$C^*=1/(1+\sqrt K)$ along with the characterization of $\pi^*$.
This proves that \emph{no} distribution $\pi$ with
$\operatorname{supp}(\pi)=\{z_1,\dots,z_m\}$ can satisfy
$C(\pi,\gamma)\equiv\text{const}$ unless the site set admits a spectral
duality. Contrapositively, site sets admitting no spectral duality
produce no reset-neutral distributions with full support.
\end{proof}

\begin{remark}[Structural content of the necessity proof]
\label{rem:necessity_structure}
The essential ingredient in Step~1 above is the specific identity
$B_\nu(z)=(p/q)^{a-z}A_\nu(a-z)$ of the biased random walk: it
guarantees that each $B_\nu(z_i)$ is \emph{purely} proportional to a
single $A_\nu(\hat z)$, with no residual contribution in the orthogonal
complement. For a general Markov process, an analogous ``rank-one
mod-$\nu$'' property would need to be assumed explicitly (beyond the
bare existence of a spectral decomposition) for the necessity direction
to carry over. This is a relevant observation for the extensions
discussed in Section~\ref{sec:conclusion}.
\end{remark}

\begin{remark}[Twisted symmetry and Doob-type structure]
\label{rem:twisted_doob}
The spectral duality condition~\eqref{eq:spectral_duality}
is a \emph{twisted symmetry}: the involution $\sigma$ acts
on sites while $\kappa$ acts on spectral channels, and
together they exchange the ruin and survival components
up to a $\nu$-independent weight. This structure is
\emph{reminiscent} of the Doob $h$-transform, which reweights
trajectories of a Markov chain by a harmonic function $h$ to
condition on a given absorption event. We emphasize, however,
that the specific $\kappa(z)=(p/q)^{a-z}$ arising here does not
coincide with the ratio $h(\sigma(z))/h(z)$ for the Doob
symmetrization $h_D(z)=(q/p)^{z/2}$ nor for the classical
harmonic function $h_C(z)=1-(q/p)^z$ of the walk conditioned
on absorption at $a$; it emerges instead as a direct consequence of the
closed-form identity~\eqref{eq:B_from_A_rw}, which relates the
two boundary-flux channels with a mode-independent proportionality
constant. That constant has a direct path-wise meaning: a path from $z$
absorbed at $a$ at step $k$ makes $a-z$ more steps to the right than to
the left, and its reflection $x\mapsto a-x$ is a path from $a-z$
absorbed at $0$ at step $k$ with the two kinds of steps exchanged, so
the two probabilities differ exactly by the factor
$(p/q)^{a-z}$~\cite{VegaCoso2026}. This is why the identity holds step
by step in time as well as mode by mode, which is what
Proposition~\ref{prop:any_mechanism} below uses. The broader connection between spectral dualities of
the form~\eqref{eq:spectral_duality} and Doob-type reweightings
in more general Markov chains --- that is, the identification
of generic conditions under which $B_\nu$ and $A_\nu$ are linked
by a site-dependent but mode-independent factor --- constitutes
an interesting question for future investigation.
\end{remark}

\begin{remark}[The neutral site as spectral fixed point]
\label{rem:neutral_site}
A spectrally neutral site $z_0$ (with $\sigma(z_0)=z_0$)
is characterized not by any geometric property but by the
spectral condition $\kappa(z_0)=\sqrt{K}$: its spectral
coefficient must equal the geometric mean of the pair
coefficients. This is the condition under which $z_0$
lies on the ``fixed line'' of the twisted symmetry ---
the direction in the $(A_\nu,B_\nu)$ plane that is
invariant under the action of $\kappa$. The weight
$\pi_{z_0}^*$ is free precisely because adding mass at
$z_0$ does not move the projective fixed point $[C^*:1-C^*]$.
\end{remark}

\subsection{Application to the biased random walk}
\label{sec:application}

We now specialize Theorem~\ref{thm:general} to the biased
random walk on $\{0,\dots,a\}$. The spectral coefficients
$A_\nu(z)$ and $B_\nu(z)$ are given
by~\eqref{eq:A_nu}--\eqref{eq:B_nu}, and satisfy the
key identity
\begin{equation}
B_\nu(z) = \left(\frac{p}{q}\right)^{a-z} A_\nu(a-z)
\qquad \forall\,\nu=1,\dots,a-1,
\label{eq:B_from_A_rw}
\end{equation}
which is a direct computation from the definitions.
Hypotheses~\ref{hyp:f_indep} and~\ref{hyp:A_indep}
are verified in Section~\ref{sec:spectral}: the functions
$f_\nu(\gamma)=\sqrt{pq}\,(1-\gamma)/(1-\lambda_\nu(1-\gamma))$
are linearly independent because the $\lambda_\nu$ are pairwise
distinct, and the
matrix $(A_\nu(z_i))_{\nu,i}$ has full column rank by the
orthogonality of the eigenfunctions $\{\psi_\nu\}$.

\begin{definition}[Symmetric reset configuration]
\label{def:symmetric}
A set of reset sites $\{z_1,\dots,z_m\}\subset\{1,\dots,a-1\}$
is called \emph{symmetric} if it admits a spectral duality
in the sense of Definition~\ref{def:spectral_duality} with
the involution $\sigma(z)=a-z$ and weights
$\kappa(z)=(p/q)^{a-z}$. Explicitly: the sites can be
partitioned into pairs $(z_i,z_i')$ with $z_i+z_i'=a$,
together with at most one site at $z=a/2$ when $a$ is even.
\end{definition}

\begin{corollary}[Critical distribution for the biased random walk]
\label{cor:rw}
For the biased random walk on $\{0,\dots,a\}$, the spectral
duality condition~\eqref{eq:spectral_duality} is equivalent
to the geometric symmetry $z_i+z_i'=a$. For any symmetric
set of reset sites, hypotheses~\ref{hyp:C_consistent}
and~\ref{hyp:neutral_consistent} are automatically satisfied
with $K=(p/q)^a$, and Theorem~\ref{thm:general} gives a
family of critical distributions $\pi^*$ characterized by
\begin{equation}
\frac{\pi_{z_i}^*}{\pi_{z_i'}^*}
= \left(\frac{q}{p}\right)^{a/2-z_i},
\qquad i=1,\dots,k,
\label{eq:pi_star_rw}
\end{equation}
with $\pi_{a/2}^*$ completely free if the midpoint is present.
The universal invariant value is
\begin{equation}
C^* = \frac{(q/p)^{a/2}}{1+(q/p)^{a/2}}
= q_{a/2}^{(0)},
\label{eq:C_star_rw}
\end{equation}
where $q_{a/2}^{(0)}$ is the classical ruin probability
from $z=a/2$~\cite{Feller1968, Redner2001}, defined by
analytic continuation for odd $a$.
Conversely, if the sites are not symmetric, no distribution
$\pi$ with full support on them satisfies $C(\pi,\gamma)=\mathrm{const}$
for all $\gamma$.
\end{corollary}

\begin{proof}
Identity~\eqref{eq:B_from_A_rw} shows that $\sigma(z)=a-z$
and $\kappa(z)=(p/q)^{a-z}$ satisfy~\eqref{eq:spectral_duality}
for any pair with $z+z'=a$. Conversely, for $z+z'\neq a$,
the ratio $B_\nu(z)/A_\nu(z')$ involves
$\sin(\pi\nu(a-z)/a)/\sin(\pi\nu z'/a)$, which depends
on $\nu$ unless $a-z=z'$. Hence the spectral duality
is equivalent to $z+z'=a$.

For hypothesis~\ref{hyp:C_consistent}:
$\kappa(z_i)\kappa(z_i')=(p/q)^{a-z_i}(p/q)^{z_i}=(p/q)^a=K$,
the same for all pairs. For hypothesis~\ref{hyp:neutral_consistent}:
$\kappa(a/2)=(p/q)^{a/2}=\sqrt{(p/q)^a}=\sqrt{K}$.

Substituting $K=(p/q)^a$ into~\eqref{eq:C_star_general}
gives~\eqref{eq:C_star_rw}, and substituting into
~\eqref{eq:pi_star_general} gives~\eqref{eq:pi_star_rw}.
The identity $C^*=q_{a/2}^{(0)}$ follows by factoring
$(1-(q/p)^a)=(1-(q/p)^{a/2})(1+(q/p)^{a/2})$ in the
classical ruin formula~\cite{Feller1968}.
\end{proof}

\begin{corollary}[Two-site case]
\label{cor:two_site}
Let $m=2$ with symmetric sites $z_1$ and $z_2=a-z_1$.
The unique critical distribution is
\begin{equation}
\pi_1^* = \frac{(q/p)^{a/2-z_1}}{1+(q/p)^{a/2-z_1}},
\qquad
\pi_2^* = \frac{1}{1+(q/p)^{a/2-z_1}},
\label{eq:pi_star_two_site}
\end{equation}
and the invariant coupling constant is $C^*=q_{a/2}^{(0)}$.
\end{corollary}

\begin{proof}
With $m=2$, the ratio~\eqref{eq:pi_star_rw} gives
$\pi_1^*/\pi_2^*=(q/p)^{a/2-z_1}$. Together with
$\pi_1^*+\pi_2^*=1$ this yields~\eqref{eq:pi_star_two_site}
uniquely.
\end{proof}

\begin{remark}[Spatial symmetry as manifestation of spectral duality]
\label{rem:spatial_spectral}
For the biased random walk, the geometric condition
$z_i+z_i'=a$ is the exact condition under which the
spectral duality~\eqref{eq:spectral_duality} holds with
$\nu$-independent weights. The spatial symmetry is
therefore not the fundamental reason for the existence
of $\pi^*$ --- it is the geometric form taken by the
spectral duality in this particular model. In a different
process, the spectrally dual sites might not be
geometrically symmetric, but the criterion of
Theorem~\ref{thm:general} would still apply.
\end{remark}

\begin{remark}[Validity for odd $a$]
\label{rem:odd_a}
Theorem~\ref{thm:general} and Corollary~\ref{cor:rw}
hold for any $a\geq 2$, whether even or odd. For odd $a$,
no integer midpoint exists: all symmetric sites appear
in strict pairs and no neutral site is present. The
formula $C^*=q_{a/2}^{(0)}$ remains valid by analytic
continuation, with $a/2$ not an integer. Numerical
validation for $a=9$, $p=0.7$, sites $\{3,6\}$ gives
$\pi_1^*=(3/7)^{3/2}/(1+(3/7)^{3/2})=0.21910\dots$
and $C^*=q_{4.5}^{(0)}=0.02161\dots$, both confirmed
to within $10^{-10}$ uniformly over $\gamma\in(0.05,0.95)$.
\end{remark}

\begin{remark}[Numerical verification]
\label{rem:verification}
Formulas~\eqref{eq:pi_star_rw}--\eqref{eq:C_star_rw} are
confirmed to machine precision by systematic computation.
For $a=10$, sites $\{3,7\}$: $p=0.6$ gives $\pi_1^*=4/13$;
$p=0.7$ gives $\pi_1^*=9/58$; $p=0.5$ gives $\pi_1^*=1/2$.
For sites $\{2,8\}$, $p=0.6$: $\pi_1^*=8/35$.
For $m=3$, sites $\{3,5,7\}$, $p=0.7$: $\pi_3^*=9/58$,
$\pi_7^*=49/58$, $\pi_5^*$ free.
In all cases $|\partial C/\partial\gamma|<10^{-11}$
uniformly over $\gamma\in(0.05,0.95)$.
\end{remark}

\begin{remark}[Generating function perspective]
\label{rem:gf}
The result admits an alternative formulation in terms of
probability generating functions~\cite{Feller1968}. For the walk
without resetting define
$U_z(s)=\mathbb{E}_{z}[s^\tau\mathbf{1}_{\{\tau_0<\tau_a\}}]$ and
$V_z(s)=\mathbb{E}_{z}[s^\tau\mathbf{1}_{\{\tau_a<\tau_0\}}]$, so that
$u(z)=U_z(1-\gamma)$ and $s(z)=U_z(1-\gamma)+V_z(1-\gamma)$. By the
linear independence of the $f_\nu$, the spectral duality with
$\nu$-independent weights is equivalent to
$V_z(s)=\kappa(z)\,U_{\sigma(z)}(s)$ for every $s\in(0,1)$, that is, to
the ratio $V_z(s)/U_{\sigma(z)}(s)$ being independent of $s$; and,
expanding in powers of $s$, to the time-domain identity
$\PP_z(\tau_a=k,\,\tau_a<\tau_0)=\kappa(z)\,\PP_{\sigma(z)}(\tau_0=k,\,\tau_0<\tau_a)$
for every $k\ge1$. For the biased walk this is the step-by-step channel
duality of Ref.~\cite{VegaCoso2026}.
\end{remark}

\begin{remark}[Connection with single-site invariance]
\label{rem:connection_paper1}
Under $\pi=\pi^*$, the ruin probability satisfies
$q_z(\gamma)=u(z;\gamma)+(1-s(z;\gamma))C^*$ with
$C^*=q_{a/2}^{(0)}$ independent of $\gamma$. At
$z=a/2$ and even $a$, the channel duality gives
$u(a/2)=C^*s(a/2)$, hence $q_{a/2}(\gamma)=C^*$ is itself
independent of $\gamma$, as in the midpoint
invariance of~\cite{VegaCoso2026}. The present result
generalizes this from a property of a single starting
point to a structural property of the reset distribution:
$\pi^*$ is the distribution that makes the global
coupling constant equal to the classical ruin probability
from the midpoint, for all resetting rates simultaneously.
\end{remark}

\subsection{Independence of the resetting mechanism}
\label{sec:any_mechanism}

Geometric resetting enters the analysis above only through the
discount $(1-\gamma)^k$ of the reset-free absorption laws. Because the
duality~\eqref{eq:B_from_A_rw} holds step by step in time, the critical
distributions are in fact neutral for every resetting mechanism, by the
argument of Ref.~\cite{VegaCoso2026} for a single reset site. Let
$\mathcal R\subset\{1,2,\dots\}$ be a set of reset times, finite or
infinite, deterministic or random, independent of the steps
$\{\xi_n\}$; consider the walk whose initial position is drawn from a
distribution $\mu$ on $\{1,\dots,a-1\}$ and which, at each
$n\in\mathcal R$, is relocated (without moving) to an independent draw
from $\pi$, as in Definition~\ref{def:rw_multi_reset} with
$\{R_n=1\}$ replaced by $\{n\in\mathcal R\}$. Geometric resetting is
the case in which $\mathcal R$ is a Bernoulli$(\gamma)$ set; the ruin
probability is then $q_{z}(\gamma)$ when $\mu=\delta_z$, and
$\sum_i\pi_{z_i}q_{z_i}(\gamma)=C(\pi,\gamma)$ when $\mu=\pi$. Write
$\tau$ for the absorption time and $N$ for the number of resets before
absorption.

We call a distribution $\mu$ on $\{1,\dots,a-1\}$ \emph{critical} if
its support is a symmetric set and
$\mu_z/\mu_{a-z}=(q/p)^{a/2-z}$ for every $z$ in it, as
in~\eqref{eq:pi_star_rw}. Two cases are singled out: $\mu=\pi^*$, the
reset distribution itself; and, for even $a$, $\mu=\delta_{a/2}$,
the walk started from the midpoint.

\begin{proposition}[Neutrality for every resetting mechanism]
\label{prop:any_mechanism}
Let $\{z_1,\dots,z_m\}$ be a symmetric set of reset sites and $\pi^*$ a
critical distribution of Corollary~\ref{cor:rw}. For the walk started
from a critical distribution $\mu$ --- in particular from $\pi^*$, or,
for even $a$, from the midpoint $a/2$ --- and reset to $\pi^*$ at the
times of an arbitrary resetting mechanism $\mathcal R$:
\begin{enumerate}[label=(\roman*)]
\item for every $n\ge1$ and $\ell\ge0$, conditionally on $\mathcal R$ and
hence also unconditionally,
$\PP(\tau=n,\,N=\ell,\,X_\tau=0)=(q/p)^{a/2}\,\PP(\tau=n,\,N=\ell,\,X_\tau=a)$;
\item the side of absorption is independent of the absorption time and
of the number of resets;
\item if absorption occurs almost surely, the ruin probability is
$C^*=q^{(0)}_{a/2}$, whatever the resetting mechanism.
\end{enumerate}
\end{proposition}

\begin{proof}
Write $u_{z,k}=\PP_z(\tau_0=k,\,\tau_0<\tau_a)$ and
$v_{z,k}=\PP_z(\tau_a=k,\,\tau_a<\tau_0)$ for the walk without
resetting. The duality~\eqref{eq:B_from_A_rw}, inserted in the spectral
representation, gives $v_{z,k}=(p/q)^{a-z}u_{a-z,k}$ for every
$k\ge1$~\cite{VegaCoso2026}. Hence, for every critical distribution
$\rho$, since its support is closed under $z\mapsto a-z$,
\[
\sum_{z}\rho_z v_{z,k}
=\sum_{w}\rho_{a-w}\Bigl(\frac pq\Bigr)^{w}u_{w,k}
=\Bigl(\frac pq\Bigr)^{a/2}\sum_{w}\rho_{w}\,u_{w,k},
\]
because the defining ratio is equivalent to
$\rho_{a-w}(p/q)^{w}=(p/q)^{a/2}\rho_w$ for every site $w$ (trivially
so at $w=a/2$). Thus the $\rho$-mixtures of the two channels are
proportional at every step $k$, with the same constant $(q/p)^{a/2}$
for every critical $\rho$; we use this for $\rho=\mu$ and
$\rho=\pi^*$. Now fix a realization of $\mathcal R$;
the steps remain i.i.d.\ under this conditioning. If $n\in\mathcal R$,
both sides of (i) vanish. Otherwise let $t_\ell$ be the last reset time
before $n$ ($t_0=0$). The first cycle starts from $\mu$ and the later
ones afresh from independent draws of $\pi^*$, so the event
$\{\tau=n,\,N=\ell,\,X_\tau=0\}$ has conditional probability
$\prod_{j=1}^{\ell}\varsigma_j\cdot\sum_z\rho_z u_{z,n-t_\ell}$, with
$\rho=\mu$ if $\ell=0$ and $\rho=\pi^*$ if $\ell\ge1$, where $\varsigma_j$ is
the probability that the $j$-th cycle ends in a reset rather than in
absorption; replacing $u$ by $v$ gives the
probability of $\{\tau=n,\,N=\ell,\,X_\tau=a\}$. By the proportionality of
the mixtures the two differ by the factor $(q/p)^{a/2}$, and averaging
over $\mathcal R$ preserves the identity. This is (i); summing over $\ell$,
respectively over $n$, gives (ii), and summing over both gives
$\PP(X_\tau=0)=(q/p)^{a/2}\PP(X_\tau=a)$, which is (iii) when the two
probabilities add up to one, since
$(q/p)^{a/2}/(1+(q/p)^{a/2})=C^*$.
\end{proof}

Under geometric resetting, part (iii) with $\mu=\pi^*$ is the statement
$C(\pi^*,\gamma)=C^*$ of Corollary~\ref{cor:rw}, and with
$\mu=\delta_{a/2}$ it is the invariance $q_{a/2}(\gamma)=C^*$ of
Remark~\ref{rem:connection_paper1}, that is, the common crossing point
of Figure~\ref{fig:qz_vs_z}; the proposition shows that neither depends
on the law of the reset times. The single-site invariance of
Ref.~\cite{VegaCoso2026} is the case $\mu=\pi^*=\delta_{a/2}$, and the
weight $\pi^*_{a/2}$ of a neutral site remains free for every
mechanism. From a non-critical starting point, by contrast, the ruin
probability depends on the mechanism even when the walk is reset to
$\pi^*$. For odd $a$ there is no midpoint site and all reset sites come
in pairs, but the proportionality of the mixtures, and with it the
proposition for $\mu=\pi^*$, holds unchanged.

\newpage

\section{Numerical illustrations}
\label{sec:numerics}

The theoretical results of Section~\ref{sec:critical} are
confirmed by systematic numerical computation. Rather than
exhaustive parameter sweeps, we select three complementary
illustrations that together exhibit the key structural
features of Theorem~\ref{thm:general} and
Corollary~\ref{cor:rw}: the existence and uniqueness of the
critical distribution $\pi^*$, the universal crossing of
ruin probability curves under $\pi^*$, and the universality
of $C^*$ across symmetric site configurations.

\subsection{The critical distribution as separatrix}
\label{subsec:num_separatrix}

Figure~\ref{fig:C_vs_gamma} shows the coupling constant
$C(\pi,\gamma)$ as a function of the resetting rate $\gamma$
for five reset distributions $\pi$ on sites $\{3,7\}$, with
$a=10$ and $p=0.6$. The critical distribution predicted by
Corollary~\ref{cor:two_site} is $\pi_1^*=4/13$.

The figure reveals the phase-like structure of the space of
reset distributions: for $\pi_1<\pi_1^*$, the coupling
constant is strictly decreasing in $\gamma$; for $\pi_1>\pi_1^*$,
it is strictly increasing; and at $\pi_1=\pi_1^*$, it is
exactly constant at $C^*=q_5^{(0)}\approx 0.1164$ for all
$\gamma\in(0,1)$. The critical distribution $\pi^*$ thus
acts as a \emph{separatrix} in the simplex of reset measures.

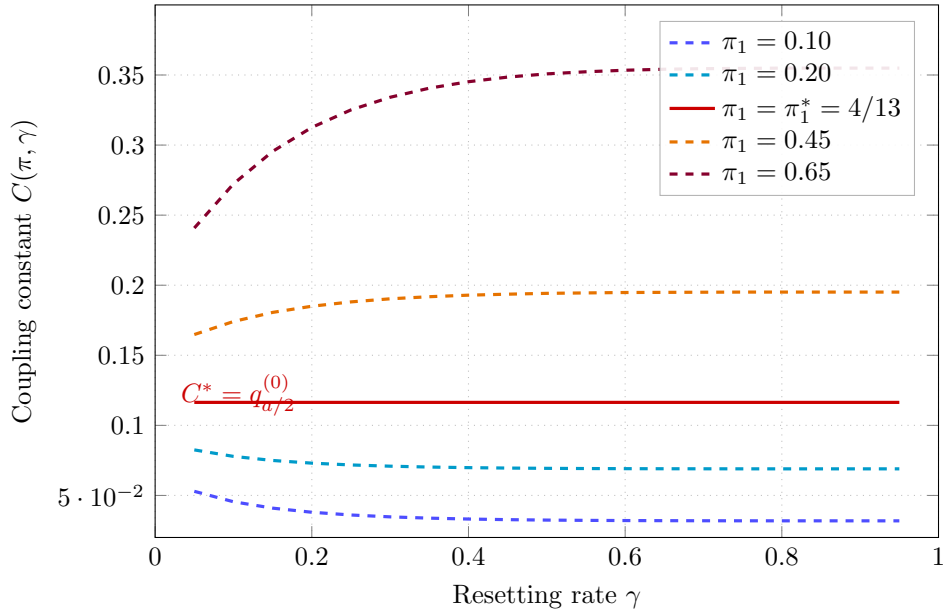
\begin{figure}[h!]
\centering
\begin{tikzpicture}
\begin{axis}[
    width  = 0.72\textwidth,
    height = 0.52\textwidth,
    xlabel = {Resetting rate $\gamma$},
    ylabel = {Coupling constant $C(\pi,\gamma)$},
    xmin = 0.0,  xmax = 1.0,
    ymin = 0.02, ymax = 0.40,
    xtick = {0.0, 0.2, 0.4, 0.6, 0.8, 1.0},
    ytick = {0.05, 0.10, 0.15, 0.20, 0.25, 0.30, 0.35},
    grid = major,
    grid style = {dotted, gray!50},
    tick label style = {font=\small},
    label style     = {font=\small},
    legend style    = {
        font        = \small,
        at          = {(0.97,0.97)},
        anchor      = north east,
        cells       = {anchor=west},
        draw        = gray!60,
        fill        = white,
        fill opacity= 0.9,
        text opacity= 1
    },
    every axis plot post/.append style = {line width=1.2pt}
]
\addplot[color=blue!70, dashed] coordinates {
(0.05,0.05294)(0.1,0.04549)(0.15,0.04093)
(0.2,0.03800)(0.25,0.03604)(0.3,0.03471)
(0.35,0.03379)(0.4,0.03316)(0.45,0.03272)
(0.5,0.03241)(0.55,0.03221)(0.6,0.03207)
(0.65,0.03199)(0.7,0.03193)(0.75,0.03190)
(0.8,0.03188)(0.85,0.03188)(0.9,0.03187)
(0.95,0.03187)};
\addlegendentry{$\pi_1 = 0.10$}
\addplot[color=cyan!80!black, dashed] coordinates {
(0.05,0.08249)(0.1,0.07786)(0.15,0.07495)
(0.2,0.07304)(0.25,0.07175)(0.3,0.07087)
(0.35,0.07025)(0.4,0.06983)(0.45,0.06953)
(0.5,0.06933)(0.55,0.06919)(0.6,0.06910)
(0.65,0.06904)(0.7,0.06901)(0.75,0.06898)
(0.8,0.06897)(0.85,0.06897)(0.9,0.06897)
(0.95,0.06897)};
\addlegendentry{$\pi_1 = 0.20$}
\addplot[color=red!80!black, thick, solid] coordinates {
(0.05,0.1164)(0.1,0.1164)(0.15,0.1164)
(0.2,0.1164)(0.25,0.1164)(0.3,0.1164)
(0.35,0.1164)(0.4,0.1164)(0.45,0.1164)
(0.5,0.1164)(0.55,0.1164)(0.6,0.1164)
(0.65,0.1164)(0.7,0.1164)(0.75,0.1164)
(0.8,0.1164)(0.85,0.1164)(0.9,0.1164)
(0.95,0.1164)};
\addlegendentry{$\pi_1 = \pi_1^* = 4/13$}
\addplot[color=orange!90!black, dashed] coordinates {
(0.05,0.1648)(0.1,0.1741)(0.15,0.1806)
(0.2,0.1850)(0.25,0.1881)(0.3,0.1903)
(0.35,0.1918)(0.4,0.1929)(0.45,0.1936)
(0.5,0.1942)(0.55,0.1945)(0.6,0.1948)
(0.65,0.1949)(0.7,0.1950)(0.75,0.1951)
(0.8,0.1951)(0.85,0.1951)(0.9,0.1951)
(0.95,0.1951)};
\addlegendentry{$\pi_1 = 0.45$}
\addplot[color=purple!70!black, dashed] coordinates {
(0.05,0.2408)(0.1,0.2723)(0.15,0.2955)
(0.2,0.3126)(0.25,0.3251)(0.3,0.3341)
(0.35,0.3406)(0.4,0.3452)(0.45,0.3485)
(0.5,0.3508)(0.55,0.3523)(0.6,0.3534)
(0.65,0.3541)(0.7,0.3545)(0.75,0.3547)
(0.8,0.3549)(0.85,0.3549)(0.9,0.3549)
(0.95,0.3549)};
\addlegendentry{$\pi_1 = 0.65$}
\node[font=\small, text=red!80!black, anchor=west]
    at (axis cs: 0.02, 0.122) {$C^* = q_{a/2}^{(0)}$};
\end{axis}
\end{tikzpicture}
\caption{Coupling constant $C(\pi,\gamma)$ as a function of
$\gamma$ for five reset distributions on sites $\{3,7\}$,
with $a=10$ and $p=0.6$. The critical distribution
$\pi_1^*=4/13$ (Corollary~\ref{cor:two_site}, red solid line)
produces a perfectly flat curve at $C^*=q_5^{(0)}\approx 0.1164$
for all $\gamma$. Distributions with $\pi_1<\pi_1^*$ give
decreasing $C(\pi,\gamma)$; those with $\pi_1>\pi_1^*$ give
increasing $C(\pi,\gamma)$. The critical distribution acts
as a separatrix in the space of reset measures.}
\label{fig:C_vs_gamma}
\end{figure}

\begin{remark}[Separatrix structure and asymptotic regime]
\label{rem:separatrix}
Figure~\ref{fig:C_vs_gamma} reveals that for each fixed $\pi$,
the limit $\lim_{\gamma\to 1}C(\pi,\gamma)$ exists and is
finite. The curves with $\pi_1<\pi_1^*$ decrease monotonically
to a lower asymptote, while those with $\pi_1>\pi_1^*$ increase
monotonically to an upper asymptote. The critical distribution
$\pi^*$ is the unique distribution already at its asymptotic
value for all $\gamma$ — in this sense, $\pi^*$ is the
distribution for which the system is ``already at equilibrium''
with respect to variations in the resetting rate. The analysis
of these asymptotic values and the interpretation of $\pi^*$
as a global fixed point of the $\gamma$-flow constitute natural
directions for future investigation.
\end{remark}


\subsection{Universal crossing under $\pi^*$}
\label{subsec:num_crossing}

Figure~\ref{fig:qz_vs_z} shows the ruin probability
$q_z(\gamma)$ as a function of the starting position $z$
for the critical distribution $\pi^*=(4/13,\,9/13)$ on
sites $\{3,7\}$, with $a=10$ and $p=0.6$. Five values of
$\gamma$ are shown, including $\gamma=0$ (classical ruin
without resetting).

All curves intersect exactly at $z=a/2=5$, where
$q_5(\gamma)=C^*=q_5^{(0)}\approx 0.1164$ for all $\gamma$,
confirming the connection between the invariant coupling
constant and the midpoint invariance of~\cite{VegaCoso2026}.
For $z<a/2$, increasing $\gamma$ reduces the ruin
probability; for $z>a/2$, it increases it.

\begin{figure}[h!]
\centering
\begin{tikzpicture}
\begin{axis}[
    width  = 0.72\textwidth,
    height = 0.52\textwidth,
    xlabel = {Starting position $z$},
    ylabel = {Ruin probability $q_z(\gamma)$},
    xmin = 0.5,  xmax = 9.5,
    ymin = 0.0,  ymax = 0.72,
    xtick = {1,2,3,4,5,6,7,8,9},
    ytick = {0.0,0.1,0.2,0.3,0.4,0.5,0.6,0.7},
    grid = major,
    grid style  = {dotted, gray!50},
    tick label style = {font=\small},
    label style      = {font=\small},
    legend style = {
        font        = \small,
        at          = {(0.97,0.97)},
        anchor      = north east,
        cells       = {anchor=west},
        draw        = gray!60,
        fill        = white,
        fill opacity= 0.9,
        text opacity= 1
    },
    every axis plot post/.append style = {line width=1.2pt}
]
\addplot[color=black, solid, mark=square*, mark size=1.8pt] coordinates {
(1,0.6608)(2,0.4346)(3,0.2839)(4,0.1834)
(5,0.1164)(6,0.07169)(7,0.04191)(8,0.02206)
(9,0.008824)};
\addlegendentry{$\gamma=0$ (classical)}
\addplot[color=blue!70, solid, mark=*, mark size=1.8pt] coordinates {
(1,0.4644)(2,0.2524)(3,0.1678)(4,0.1328)
(5,0.1164)(6,0.1054)(7,0.09350)(8,0.07604)
(9,0.04761)};
\addlegendentry{$\gamma=0.2$}
\addplot[color=cyan!80!black, solid, mark=triangle*, mark size=1.8pt] coordinates {
(1,0.3508)(2,0.1785)(3,0.1327)(4,0.1203)
(5,0.1164)(6,0.1138)(7,0.1091)(8,0.09795)
(9,0.07005)};
\addlegendentry{$\gamma=0.4$}
\addplot[color=orange!90!black, solid, mark=diamond*, mark size=1.8pt] coordinates {
(1,0.2636)(2,0.1409)(3,0.1204)(4,0.1170)
(5,0.1164)(6,0.1159)(7,0.1145)(8,0.1091)
(9,0.08727)};
\addlegendentry{$\gamma=0.6$}
\addplot[color=red!80!black, solid, mark=pentagon*, mark size=1.8pt] coordinates {
(1,0.1877)(2,0.1221)(3,0.1168)(4,0.1164)
(5,0.1164)(6,0.1163)(7,0.1162)(8,0.1147)
(9,0.1023)};
\addlegendentry{$\gamma=0.8$}
\addplot[color=gray!70, dashed, thin, forget plot]
    coordinates {(5,0.0)(5,0.72)};
\addplot[color=gray!50, dashed, thin, forget plot]
    coordinates {(0.5, 0.1163636364)(9.5, 0.1163636364)};
\addplot[color=black, only marks, mark=*, mark size=3pt, forget plot]
    coordinates {(5,0.1164)};
\node[font=\small, anchor=south west]
    at (axis cs: 5.1, 0.1163636364) {$q_5^{(0)}=C^*$};
\node[font=\footnotesize, text=blue!60, anchor=north]
    at (axis cs: 3, -0.015) {$z_1$};
\node[font=\footnotesize, text=blue!60, anchor=north]
    at (axis cs: 7, -0.015) {$z_2$};
\end{axis}
\end{tikzpicture}
\caption{Ruin probability $q_z(\gamma)$ as a function of
starting position $z$, under the critical distribution
$\pi^*=(4/13,\,9/13)$ on sites $\{z_1,z_2\}=\{3,7\}$,
with $a=10$ and $p=0.6$. All curves cross exactly at
$z=a/2=5$, where $q_5(\gamma)=C^*=q_5^{(0)}\approx 0.1164$
for all $\gamma$ (Corollary~\ref{cor:rw}). The reset
sites $z_1=3$ and $z_2=7$ are indicated on the horizontal
axis.}
\label{fig:qz_vs_z}
\end{figure}
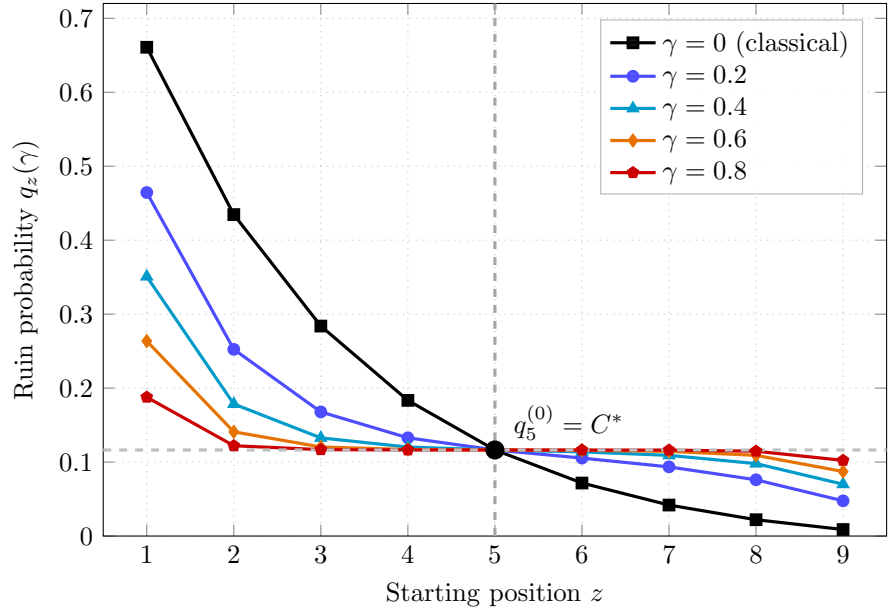


\subsection{Universality of $C^*$}
\label{subsec:num_universality}

Figure~\ref{fig:Cstar_vs_p} illustrates the universality
of $C^*$ predicted by Corollary~\ref{cor:rw}: the invariant
coupling constant depends only on $p$ and $a$, not on the
specific choice of symmetric reset sites. Four configurations
— $\{1,9\}$, $\{2,8\}$, $\{3,7\}$, $\{4,6\}$ — are
compared against the theoretical curve
$C^*(p)=q_{a/2}^{(0)}(p)$ for $a=10$. All four collapse
onto the same universal curve to within numerical precision.

\begin{figure}[H]
\centering
\begin{tikzpicture}
\begin{axis}[
    width  = 0.72\textwidth,
    height = 0.52\textwidth,
    xlabel = {Bias parameter $p$},
    ylabel = {Invariant coupling constant $C^*$},
    xmin = 0.08, xmax = 0.92,
    ymin = 0.0,  ymax = 1.05,
    xtick = {0.1,0.2,0.3,0.4,0.5,0.6,0.7,0.8,0.9},
    ytick = {0.0,0.2,0.4,0.6,0.8,1.0},
    grid = major,
    grid style  = {dotted, gray!50},
    tick label style = {font=\small},
    label style      = {font=\small},
    legend style = {
        font        = \small,
        at          = {(0.50,0.97)},
        anchor      = north,
        cells       = {anchor=west},
        draw        = gray!60,
        fill        = white,
        fill opacity= 0.9,
        text opacity= 1,
        column sep  = 6pt
    },
    every axis plot post/.append style = {line width=1.2pt}
]
\addplot[color=black, solid, thick, domain=0.10:0.90, samples=100]
    {(((1-x)/x)^5)/(1+(((1-x)/x)^5))};
\addlegendentry{$C^*=q_{a/2}^{(0)}(p)$ (theory)}
\addplot[color=blue!70, only marks, mark=*, mark size=2.5pt] coordinates {
(0.25,0.9959)(0.3,0.9857)(0.35,0.9567)
(0.4,0.8836)(0.45,0.7317)(0.5,0.5000)
(0.55,0.2683)(0.6,0.1164)(0.65,0.04331)
(0.7,0.01425)(0.75,0.004098)};
\addlegendentry{sites $\{1,9\}$}
\addplot[color=cyan!80!black, only marks, mark=triangle*, mark size=2.5pt] coordinates {
(0.2,0.9990)(0.25,0.9959)(0.3,0.9857)
(0.35,0.9567)(0.4,0.8836)(0.45,0.7317)
(0.5,0.5000)(0.55,0.2683)(0.6,0.1164)
(0.65,0.04331)(0.7,0.01425)(0.75,0.004098)
(0.8,0.000976)};
\addlegendentry{sites $\{2,8\}$}
\addplot[color=orange!90!black, only marks, mark=diamond*, mark size=2.5pt] coordinates {
(0.1,1.0000)(0.15,0.9998)(0.2,0.9990)
(0.25,0.9959)(0.3,0.9857)(0.35,0.9567)
(0.4,0.8836)(0.45,0.7317)(0.5,0.5000)
(0.55,0.2683)(0.6,0.1164)(0.65,0.04331)
(0.7,0.01425)(0.75,0.004098)(0.8,0.000976)
(0.85,0.000171)(0.9,0.0000169)};
\addlegendentry{sites $\{3,7\}$}
\addplot[color=red!80!black, only marks, mark=square*, mark size=2.5pt] coordinates {
(0.1,1.0000)(0.15,0.9998)(0.2,0.9990)
(0.25,0.9959)(0.3,0.9857)(0.35,0.9567)
(0.4,0.8836)(0.45,0.7317)(0.5,0.5000)
(0.55,0.2683)(0.6,0.1164)(0.65,0.04331)
(0.7,0.01425)(0.75,0.004098)(0.8,0.000976)
(0.85,0.000171)(0.9,0.0000169)};
\addlegendentry{sites $\{4,6\}$}
\addplot[color=gray!60, dashed, thin, forget plot]
    coordinates {(0.5,0.0)(0.5,1.0500)};
\addplot[color=black, only marks, mark=*, mark size=3pt, forget plot]
    coordinates {(0.5,0.5000)};
\node[font=\small, anchor=south west]
    at (axis cs: 0.51,0.50) {$C^*=\tfrac{1}{2}$};
\end{axis}
\end{tikzpicture}
\caption{Invariant coupling constant $C^*$ as a function of
$p$, for four symmetric site configurations with $a=10$.
The solid curve is the theoretical prediction
$C^*=q_{a/2}^{(0)}(p)$ (Corollary~\ref{cor:rw}). All four
configurations collapse onto the same universal curve,
confirming that $C^*$ depends only on $p$ and $a$.
At $p=1/2$, $C^*=1/2$ for all configurations.}
\label{fig:Cstar_vs_p}
\end{figure}
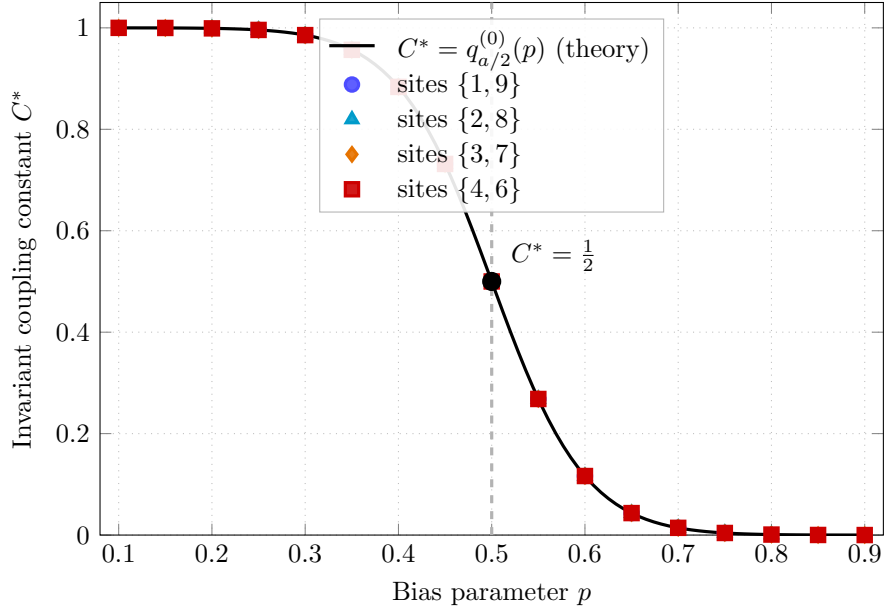


\subsection{Numerical verification of the Theorem}
\label{subsec:num_table}

Table~\ref{tab:verification} summarizes the numerical
verification of Corollary~\ref{cor:rw} for eight
representative cases covering different domain sizes
$a$, bias parameters $p$, site configurations, and
number of reset sites $m$. In each case, $\pi_1^*$ and
$C^*$ are computed from
formulas~\eqref{eq:pi_star_rw}--\eqref{eq:C_star_rw},
and the numerical value of $C^*$ is obtained by bisection
at $\gamma=0.5$. The errors are below $10^{-9}$ in all
non-degenerate cases, confirming the theoretical predictions
to high precision across all parameter regimes.

\begin{table}[H]
\centering
\caption{Numerical verification of
Corollary~\ref{cor:rw} for eight cases. The error
$|C^*_{\mathrm{num}}-C^*_{\mathrm{th}}|$ confirms the
theoretical predictions to high precision in every case;
$C^*_{\mathrm{num}}$ is obtained by bisection on $\pi_1$
at $\gamma=0.5$ in double-precision arithmetic. The case $m=3$
with sites $\{3,5,7\}$ illustrates the spectrally neutral
midpoint site ($\pi_5^*$ free, set to zero here). The case
$a=9$ validates the formula for odd $a$ via analytic
continuation. ($^\dagger$) For $p=1/2$, $C^*=1/2$ exactly
by symmetry; the bisection is degenerate since $dC/d\gamma=0$
for all symmetric $\pi$, and the reported error reflects
bisection tolerance rather than theoretical error.}
\label{tab:verification}
\smallskip
\begin{tabular}{ccccccc}
\toprule
$a$ & $p$ & $m$ & Sites & $\pi_1^*$ &
$C^*$ & Error \\
\midrule
10 & 0.5 & 2 & $\{3,7\}$   & $1/2$          & $1/2$ (exact)  & $<5\times10^{-7}$$^\dagger$ \\
10 & 0.6 & 2 & $\{3,7\}$   & $4/13$         & $0.1163636364$ & $8.53\times10^{-10}$ \\
10 & 0.6 & 2 & $\{2,8\}$   & $8/35$         & $0.1163636364$ & $1.57\times10^{-9}$ \\
10 & 0.7 & 2 & $\{3,7\}$   & $9/58$         & $0.0142521994$ & $5.67\times10^{-10}$ \\
10 & 0.7 & 3 & $\{3,5,7\}$ & $9/58$         & $0.0142521994$ & $5.67\times10^{-10}$ \\
 9 & 0.7 & 2 & $\{3,6\}$   & $0.2190952202$ & $0.0216081357$ & $1.44\times10^{-10}$ \\
 8 & 0.6 & 2 & $\{2,6\}$   & $4/13$         & $0.1649484536$ & $2.91\times10^{-10}$ \\
12 & 0.6 & 2 & $\{4,8\}$   & $4/13$         & $0.0807061791$ & $2.76\times10^{-9}$ \\
\bottomrule
\end{tabular}
\end{table}

\begin{remark}[Rational values of $\pi^*$]
\label{rem:rational}
For parameters where $a/2-z_1$ is a positive integer and
$q/p$ is rational, formula~\eqref{eq:pi_star_rw} yields
exact rational values of $\pi_1^*$: for example,
$\pi_1^*=4/13$ for $p=0.6$, $a=10$, $z_1=3$ (since
$(2/3)^2=4/9$ and $4/9/(1+4/9)=4/13$). For odd $a$,
the exponent $a/2-z_1$ is a half-integer and $\pi_1^*$
is generally irrational, as in the case $a=9$,
$p=0.7$, $z_1=3$, where $\pi_1^*=(3/7)^{3/2}/
(1+(3/7)^{3/2})\approx 0.219095$. In both cases
the formula and the numerics agree to within $10^{-9}$.
\end{remark}

\subsection{Independent Monte Carlo validation}
\label{subsec:montecarlo}

The numerical results reported in Table~\ref{tab:verification} were
obtained by direct algebraic solution of the recurrence
relations~\eqref{eq:u_recursion}--\eqref{eq:s_recursion} for $u(z)$
and $s(z)$. To rule out the possibility that systematic errors in the
spectral derivation propagate undetected through this verification,
we additionally performed independent Monte Carlo simulations of the
full reset dynamics defined in Definition~\ref{def:rw_multi_reset}.
For each configuration, $N=10^7$ independent trajectories were
simulated from each reset site $z_i$. At every time step a Bernoulli
reset event was sampled with probability $\gamma=0.5$, and
conditionally on a reset, the new site was drawn from $\pi=\pi^*$.
Each trajectory was iterated until absorption at $0$ (ruin) or at
$a$ (escape). The empirical estimator
\[
\hat C^*_{\mathrm{MC}}
:= \sum_{i=1}^m \pi^*_{z_i}\,\hat q_{\mathrm{MC}}(z_i),
\qquad
\hat q_{\mathrm{MC}}(z) := \frac{1}{N}\sum_{n=1}^N
\mathbf{1}_{\{\tau_0^{(n)} < \tau_a^{(n)}\}},
\]
provides a fully model-independent check of the theoretical
prediction.

A particularly stringent test concerns the freedom of spectrally
neutral sites established in Theorem~\ref{thm:general}. When the
midpoint $z_0=a/2$ belongs to the reset set, its weight
$\pi^*_{a/2}\geq 0$ is completely free, leaving $C^*$ unchanged. To
test this empirically, we ran the configuration $\{3,5,7\}$ at
$a=10$, $p=0.7$ \emph{twice}, with two non-trivial midpoint weights
$\pi_5=0.3$ and $\pi_5=0.7$, while keeping the constraint
$\pi_3:\pi_7=9:49$ on the remaining mass. The theorem predicts that
both runs should yield the same $C^*_{\mathrm{th}}\approx 0.014252$
regardless of $\pi_5$. The result of this test is illustrated in
Figure~\ref{fig:pi5_freedom}.

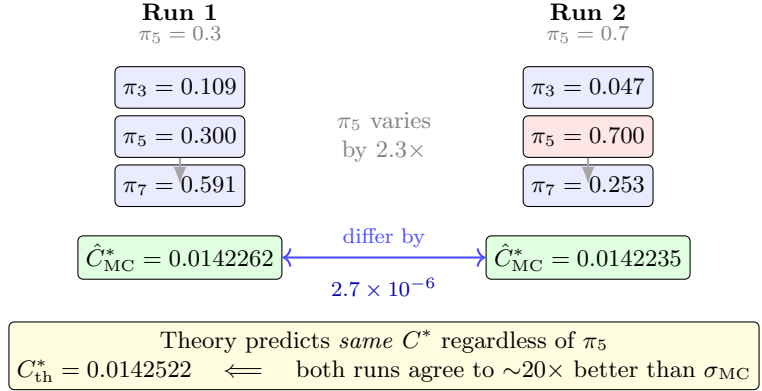
\begin{figure}[H]
\centering
\begin{tikzpicture}[
    box/.style={draw, rounded corners=2pt, minimum height=0.55cm,
                font=\footnotesize, align=center, inner sep=3pt},
    pival/.style={box, fill=blue!8, minimum width=1.6cm},
    cval/.style={box, fill=green!12, minimum width=2.2cm,
                 font=\footnotesize\bfseries},
    sublabel/.style={font=\scriptsize, align=center, gray}
]

\node[font=\small\bfseries] at (0,4.4)
    {Test of $\pi_5$ freedom: $a=10,\ p=0.7,\ \{3,5,7\}$};

\node[font=\footnotesize\bfseries] at (-2.7,3.7) {Run 1};
\node[font=\footnotesize\bfseries] at (2.7,3.7) {Run 2};
\node[sublabel] at (-2.7,3.4) {$\pi_5 = 0.3$};
\node[sublabel] at (2.7,3.4)  {$\pi_5 = 0.7$};

\node[pival] (pi3a) at (-2.7,2.7) {$\pi_3 = 0.109$};
\node[pival] (pi5a) at (-2.7,2.05) {$\pi_5 = 0.300$};
\node[pival] (pi7a) at (-2.7,1.4) {$\pi_7 = 0.591$};

\node[pival] (pi3b) at (2.7,2.7) {$\pi_3 = 0.047$};
\node[pival, fill=red!10] (pi5b) at (2.7,2.05) {$\pi_5 = 0.700$};
\node[pival] (pi7b) at (2.7,1.4) {$\pi_7 = 0.253$};

\node[font=\footnotesize, align=center, gray] at (0, 2.05)
    {$\pi_5$ varies\\by $2.3\times$};

\draw[-{Latex}, thick, gray!70] (pi5a.south) -- ++(0,-0.35);
\draw[-{Latex}, thick, gray!70] (pi5b.south) -- ++(0,-0.35);

\node[cval] (cmc1) at (-2.7,0.45) {$\hat C^*_{\mathrm{MC}} = 0.0142262$};
\node[cval] (cmc2) at (2.7,0.45)  {$\hat C^*_{\mathrm{MC}} = 0.0142235$};

\draw[<->, thick, blue!70] (cmc1.east) -- node[above, font=\scriptsize]
    {differ by} (cmc2.west);
\node[font=\scriptsize, blue!70!black] at (0,0.05)
    {$2.7 \times 10^{-6}$};

\node[box, fill=yellow!15, minimum width=8cm, font=\footnotesize]
    at (0,-0.85)
    {Theory predicts \emph{same} $C^*$ regardless of $\pi_5$\\
     $C^*_{\mathrm{th}} = 0.0142522 \quad\Longleftarrow\quad$
     both runs agree to ${\sim}20\times$ better than
     $\sigma_{\mathrm{MC}}$};

\end{tikzpicture}
\caption{Direct empirical confirmation of the freedom of spectrally
neutral sites for the $\{3,5,7\}$ configuration ($a=10$, $p=0.7$,
$\gamma=0.5$). Even though $\pi_5$ varies by a factor of $2.3$
between the two runs (and $\pi_3,\pi_7$ change accordingly), both
Monte Carlo estimates $\hat C^*_{\mathrm{MC}}$ coincide with the
theoretical value $C^*_{\mathrm{th}}=0.0142522$ to within
statistical error, and differ from each other by only
$2.7\times 10^{-6}$ ---an order of magnitude smaller than the
combined statistical error
$\sqrt{2}\,\sigma_{\mathrm{MC}}\approx 5.3\times 10^{-5}$. The
spectrally neutral midpoint weight is thus directly demonstrated to
be a gauge degree of freedom of the critical reset family.}
\label{fig:pi5_freedom}
\end{figure}

The aggregate validation across all configurations of
Table~\ref{tab:verification} is summarized in
Table~\ref{tab:mc_validation}. In every case, the absolute deviation
$|\hat C^*_{\mathrm{MC}}-C^*_{\mathrm{th}}|$ falls within the
expected statistical error
$\sigma_{\mathrm{MC}}=\sqrt{C^*(1-C^*)/N}$, and a Pearson $\chi^2$
statistic over the nine entries gives $\chi^2\approx 2.93$ for $9$
degrees of freedom, fully consistent with the null hypothesis of
agreement between Monte Carlo and theory. The agreement holds
uniformly across unbiased ($p=0.5$) and biased ($p=0.6,0.7$)
regimes, even and odd domain sizes $a$, and two- and three-site
configurations.

\begin{table}[H]
\centering
\caption{Independent Monte Carlo validation of
Corollary~\ref{cor:rw}. For each configuration, $N=10^7$ trajectories
of the full reset dynamics of Definition~\ref{def:rw_multi_reset}
were simulated at $\gamma=0.5$ with reset distribution $\pi=\pi^*$.
The two rows for $\{3,5,7\}$ correspond to two distinct non-trivial
values of the spectrally neutral weight $\pi_5$, both yielding the
same $C^*_{\mathrm{th}}$ as predicted by Theorem~\ref{thm:general}.
The expected statistical error is
$\sigma_{\mathrm{MC}}=\sqrt{C^*(1-C^*)/N}$.}
\label{tab:mc_validation}
\smallskip
\begin{tabular}{cclccc}
\toprule
$a$ & $p$ & Sites ($\pi$) &
$C^*_{\mathrm{th}}$ &
$\hat C^*_{\mathrm{MC}}$ &
$|\hat C^*_{\mathrm{MC}}-C^*_{\mathrm{th}}|$ \\
\midrule
10 & 0.5 & $\{3,7\},\ \pi^*=(\tfrac12,\tfrac12)$
   & $0.5000000000$ & $0.5000686000$ & $6.9\times 10^{-5}$ \\
10 & 0.6 & $\{3,7\},\ \pi^*=(\tfrac{4}{13},\tfrac{9}{13})$
   & $0.1163636364$ & $0.1163768231$ & $1.3\times 10^{-5}$ \\
10 & 0.6 & $\{2,8\},\ \pi^*=(\tfrac{8}{35},\tfrac{27}{35})$
   & $0.1163636364$ & $0.1162992800$ & $6.4\times 10^{-5}$ \\
10 & 0.7 & $\{3,7\},\ \pi^*=(\tfrac{9}{58},\tfrac{49}{58})$
   & $0.0142521994$ & $0.0142300776$ & $2.2\times 10^{-5}$ \\
10 & 0.7 & $\{3,5,7\},\ \pi_5=0.3$
   & $0.0142521994$ & $0.0142261862$ & $2.6\times 10^{-5}$ \\
10 & 0.7 & $\{3,5,7\},\ \pi_5=0.7$
   & $0.0142521994$ & $0.0142235005$ & $2.9\times 10^{-5}$ \\
 9 & 0.7 & $\{3,6\},\ \pi_3^*\!\approx\!0.219095$
   & $0.0216081357$ & $0.0215828418$ & $2.5\times 10^{-5}$ \\
 8 & 0.6 & $\{2,6\},\ \pi^*=(\tfrac{4}{13},\tfrac{9}{13})$
   & $0.1649484536$ & $0.1648664308$ & $8.2\times 10^{-5}$ \\
12 & 0.6 & $\{4,8\},\ \pi^*=(\tfrac{4}{13},\tfrac{9}{13})$
   & $0.0807061791$ & $0.0807359308$ & $3.0\times 10^{-5}$ \\
\bottomrule
\end{tabular}
\end{table}

\newpage

\section{Conclusion and Outlook}
\label{sec:conclusion}

In this work we have extended the theory of stochastic resetting in confined
domains to the multi-site case and established a general structural criterion
for the existence of reset-invariant distributions.

\subsection*{Summary of results}

Using renewal theory, we derived an exact closed-form expression for the
ruin probability $q_z(\gamma)$ under multi-site geometric resetting,
showing that all dependence on the reset distribution $\pi$ is mediated
through the single scalar $C(\pi,\gamma)=\bar{u}_\pi/\bar{s}_\pi$.
A spectral analysis via the Doob $h$-transform then revealed the deep
structure of this coupling constant.

Our central result (Theorem~\ref{thm:general}) establishes a general
criterion --- valid for any absorbed Markov chain admitting a spectral
decomposition --- for the existence of a critical reset distribution $\pi^*$
with $C(\pi^*,\gamma)=\mathrm{const}$. The criterion is a
\emph{spectral duality condition}: the existence of an involution $\sigma$
on the reset sites and $\nu$-independent weights $\kappa(z)$ such that
$B_\nu(z)=\kappa(z)A_\nu(\sigma(z))$ for all spectral modes $\nu$,
with a common pair product $\kappa(z)\kappa(\sigma(z))=K$ and
$\kappa=\sqrt K$ at fixed points of $\sigma$.
This is a \emph{twisted symmetry} that exchanges the ruin and survival
spectral channels with a mode-independent weight, and is intimately related
to the Doob $h$-transform duality structure.

For the biased random walk (Corollary~\ref{cor:rw}), the spectral duality
condition is equivalent to the geometric symmetry $z_i+z_i'=a$ of the reset
sites. The critical distribution satisfies
$\pi_{z_i}^*/\pi_{z_i'}^*=(q/p)^{a/2-z_i}$, and the invariant coupling
constant is
\[
C^* = \frac{(q/p)^{a/2}}{1+(q/p)^{a/2}} = q_{a/2}^{(0)},
\]
the classical ruin probability from the midpoint, independent of the specific
symmetric site configuration, the number of reset sites $m$, the domain
parity, and the resetting rate $\gamma$. Because the duality holds step
by step in time, the same value is obtained for every resetting
mechanism independent of the walk
(Proposition~\ref{prop:any_mechanism}). Numerical confirmation is
achieved to machine precision across all parameter regimes tested.

These findings reveal a previously hidden structure in reset processes,
linking spectral duality, spatial symmetry, and invariant quantities in a
unified framework that generalizes the single-site midpoint invariance
of~\cite{VegaCoso2026}.

\subsection*{Outlook}

Several directions for further investigation emerge naturally from this work;
we briefly outline two of them, which have since been taken up in
subsequent papers~\cite{PaperIII, PaperIV}.

A first direction concerns the \emph{global structure} of the coupling
constant $C(\pi,\gamma)$ in the full simplex of reset distributions, beyond
the critical set $\pi^*$ itself. Numerical evidence (Section~\ref{sec:numerics})
shows that $\pi^*$ acts as a separatrix: distributions ``below'' $\pi^*$
produce monotonically decreasing $C(\pi,\gamma)$, distributions ``above''
monotonically increasing, and $\pi^*$ exactly constant. This phase-like
decomposition of the simplex, together with the universality of $C^*$ across
different symmetric site configurations (a collapse phenomenon onto a single
curve $C^*=q_{a/2}^{(0)}(p)$ independent of the choice of pair), and the
asymptotic regime $\gamma\to 1$, raises structural questions; the
separatrix structure in the simplex of reset distributions is analysed
in Ref.~\cite{PaperIII}.

A second, more foundational, direction concerns the \emph{structural origin}
of the spectral duality identified here. Remark~\ref{rem:twisted_doob} points
towards a Doob-type reweighting, and Remark~\ref{rem:necessity_structure}
identifies a ``rank-one mod-$\nu$'' property as the key ingredient for
necessity. A natural question is whether these phenomena admit a common
explanation in terms of the resolvent of the absorbing generator and its
rank-one perturbation under resetting, connecting our results with the
classical Doob $h$-transform and Perron--Frobenius theory. This perspective,
and its implications for generic absorbing Markov chains with resetting, is
developed in Ref.~\cite{PaperIV}; the abstract structure behind the duality
is studied in Refs.~\cite{PaperV, PaperVI}.

\section*{Acknowledgments}
The author thanks Prof.\ Javier Villarroel for insightful discussions and
acknowledges support from the University of Salamanca.


\end{document}